\documentclass{amsart}

\usepackage[left=25mm,right=25mm,top=25mm,bottom=25mm]{geometry}
\usepackage{graphicx}
\usepackage{bm}
\usepackage[
	pdftitle={},
	pdfauthor={},
	ocgcolorlinks,
	linkcolor=linkblue,
	citecolor=linkred,
	urlcolor=linkred]
{hyperref}
\usepackage{color}
\definecolor{linkred}{rgb}{0.75,0,0}
\definecolor{linkblue}{rgb}{0,0,0.75}
\usepackage{mathpazo}
\usepackage{microtype}
\usepackage{multicol}
\usepackage{booktabs}
\usepackage{latexsym}
\usepackage{multirow}
\usepackage{setspace}

\theoremstyle{plain}
\newtheorem{theorem}{Theorem}
\newtheorem{proposition}{Proposition}[section]
\newtheorem{lemma}[proposition]{Lemma}

\newtheorem{corollary}[proposition]{Corollary}

\newcommand{\bt}{\begin{theorem}}
\newcommand{\et}{\end{theorem}}

\theoremstyle{definition}
\newtheorem{example}[proposition]{Example}
\newtheorem{definition}[proposition]{Definition}
\newtheorem{remark}[proposition]{Remark}
\newcommand{\beq}{\begin{equation}}
\newcommand{\eeq}{\end{equation}}
\newcommand{\bl}{\begin{lemma}}
\newcommand{\el}{\end{lemma}}

\newcommand{\cal}{\mathcal}
\newcommand{\Res}{\mathop{\,\rm Res\,}}
\newcommand{\om}{\omega}
\newcommand{\bc}{\mathbb{C}}
\newcommand{\bn}{\mathbb{N}}

\newcommand{\ca}{\mathcal{A}}
\newcommand{\cb}{\mathcal{B}}
\newcommand{\cc}{\mathcal{C}}
\newcommand{\cd}{\mathcal{D}}

\newcommand{\cp}{{\cal P}}

\newcommand{\bp}{{\mathbb P}}
\newcommand{\diag}{\mathrm{diag}}
\newcommand{\modm}{\cal M}
\newcommand{\un}{1\!\!1}

\begin{document}
	
\title{Primary Invariants of Hurwitz Frobenius Manifolds}

\author[P.~Dunin-Barkowski]{P.~Dunin-Barkowski}
\address{P.~D.-B.:
	Faculty of Mathematics, National Research University Higher School of Economics, Usacheva 6, 119048 Moscow, Russia}
\email{ptdbar@hse.ru}

\author[P.~Norbury]{P.~Norbury}
\address{P.~N.: Department of Mathematics and Statistics, University of Melbourne, Australia 3010}
\email{pnorbury@ms.unimelb.edu.au}

\author[N.~Orantin]{N.~Orantin}
\address{N.~O.:  D\'epartement de math\'ematiques, 
Ecole Polytechnique F\'ed\'erale de Lausanne,
CH-1015 Lausanne,
Switzerland}
\email{nicolas.orantin@epfl.ch}

\author[A.~Popolitov]{A.~Popolitov}
\address{A.~P.: Korteweg-de Vries Institute for Mathematics, University of Amsterdam, Postbus 94248, 1090 GE Amsterdam, The Netherlands and ITEP, Moscow, Russia}
\email{A.Popolitov@uva.nl}

\author[S.~Shadrin]{S.~Shadrin}
\address{S.~S.: Korteweg-de Vries Institute for Mathematics, University of Amsterdam, Postbus 94248, 1090 GE Amsterdam, The Netherlands}
\email{S.Shadrin@uva.nl}

\subjclass[2010]{Primary 53D45; Secondary 32G15, 81T45.}

\keywords{Frobenius manifolds, spectral curve topological recursion, Hurwitz spaces}

\begin{abstract}
Hurwitz spaces parameterizing covers of the Riemann sphere can be equipped with a Frobenius structure. In this review, we recall the construction of such Hurwitz Frobenius manifolds as well as the correspondence between semisimple Frobenius manifolds and the topological recursion formalism. We then apply this correspondence to Hurwitz Frobenius manifolds by explaining that the corresponding primary invariants can be obtained as periods of multidifferentials globally defined on a compact Riemann surface by topological recursion. Finally, we use this construction to reply to the following question in a large class of cases: given a compact Riemann surface, what does the topological recursion compute?

%
%It is a classical result that flat coordinates for a Hurwitz Frobenius manifold can be obtained as periods of a differential along cycles on the domain curve.  In this review
%
%
%
%We generalise this construction to primary invariants of the Hurwitz Frobenius manifolds.  We show that they can be obtained as periods of multidifferentials along the same cycles.  The multidifferentials are obtained via the topological recursion procedure.
\end{abstract}

\maketitle

\tableofcontents

\section{Introduction}

Consider a diagonal flat metric on a complex manifold $M$ with local coordinates $u=(u_1,...,u_n)$ 
\beq \label{diagmet}
ds^2=\sum_{i=1}^N\eta_i(u)du_i^2\quad \text{is flat},
\eeq
generated by a potential $H:M\to\bc$
\beq  \label{potmet}
\eta_i(u)=\partial_{u_i}H,\quad i=1,...,N.
\eeq
Metrics satisfying \eqref{diagmet} and \eqref{potmet} are known as Darboux-Egoroff metrics \cite{Dub94}.  Condition \eqref{diagmet} is equivalent to a nonlinear PDE in $\eta_i(u)$ which gives vanishing of the Riemann curvature tensor $R_{ijkl}=0$.  %Equivalently $ds^2$ is flat if there exists a coordinate system $(t_1,...,t_N)$ with respect to which $ds^2$ is constant.  
The PDE becomes integrable when condition \eqref{potmet} is added.  By a metric we mean a smooth family of complex non-degenerate symmetric bilinear forms on the tangent space $T_mM$, so in particular it is not Riemannian.

Analogous to K. Saito's construction \cite{SaiPer} of flat coordinates on unfolding spaces of singularities, Dubrovin \cite{Dub94} and Krichever \cite{KriAlg} produced beautiful families of Darboux-Egoroff metrics on moduli spaces of pairs $(\Sigma,x)$ consisting of an algebraic curve $\Sigma$ equipped with a meromorphic function $x:\Sigma\to\bc$.  Such a pair $(\Sigma,x)$ is a point in a Hurwitz space $H_{g,\mu}$ which parametrises covers $x:\Sigma\to\mathbb{P}^1$ of genus $g$ with points above infinity marked and with fixed ramification profile $(\mu_1,\dots,\mu_d)$.  Further, choose a symplectic basis of cycles $({\cal A}_i,{\cal B}_i)_{i=1, \dots,g}$ on $\Sigma$ to define a point in a cover $\widetilde{H}_{g,\mu}$ of a Hurwitz space.  Namely, $\widetilde{H}_{g,\mu}$ consists of the data of a point in a Hurwitz space together with the data of a Torelli marking. One goes from one sheet of the cover $\widetilde{H}_{g,\mu}$ to another one through the action of modular transformations $Sp(2g,\mathbb{Z})$. 
\begin{definition}  \label{def:gencyc}
Given $(\Sigma,x,\{ {\cal A}_i,{\cal B}_i\}_{i=1,...,g})$ %such that ${\cal A}_i,{\cal B}_i$ avoid the poles of $x$, 
define a set of generalised contours $\cd$ on $\Sigma$ as follows.  Choose representatives for $\{{\cal A}_i,{\cal B}_i\}_{i=1,...,g}$ in $H_1(\Sigma\setminus x^{-1}(\infty))$ and choose a set of relative homology classes $\gamma_i\in H_1(\Sigma,x^{-1}(\infty))$, $i=2,...,d$ such that $\gamma_i\subset\Sigma\setminus\{\ca_1,...\ca_g,\cb_1,...,\cb_g\}$ runs from $\infty_i$ to $\infty_1$ where the poles of $x$ are given by $x^{-1}(\infty)=\{\infty_1,...,\infty_d\}$ with respective orders $\{\mu_1,...,\mu_d\}$.  Let $\cc_{\infty_i}$, $i=1,...,d$ be small circles around each pole $\infty_i$ of $x$.  Then define
\beq   \label{gencon}
\cd=\{x\ca_1,...,x\ca_g,\cb_1,...,\cb_g\} \bigcup_{ i=2,...,d} \{\gamma_i,x\cc_{\infty_i}\} \bigcup_{\begin{array}{c}
k=1,...,\mu_i-1, \cr
 j=1,...,d
 \end{array}} \{x^{k/\mu_j}\cc_{\infty_j}\}.
\eeq
\end{definition}

\noindent If $x$ has only simple poles, so each $\mu_i=1$, then the contours are built out of classes in $H_1(\Sigma\setminus x^{-1}(\infty))$ and $H_1(\Sigma,x^{-1}(\infty))$.  A contour $\cc$ acts on a differential $\omega$ by $\omega\mapsto\int_{\cc}\omega$, and by $x\cc$ we mean $\omega\mapsto\int_{x\cc}\omega:=\int_{\cc}x\omega$.  We often enumerate the elements of $\cd$ by $\cc_{\alpha}\in\cd$ for $\alpha=1,...,N=|\cd|$.  Note that $N=\dim H_{g,\mu}$---see \eqref{flat=can}.

%Note that $\cd$ depends on the Torelli marking.
\begin{definition}  \label{bergman}
On any compact Riemann surface $(\Sigma,\{ {\cal A}_i\}_{i=1,...,g})$ with a given set of $\cal A$-cycles, define a {\em Bergman kernel} $B(p,p')$ to be a
symmetric bidifferential, i.e. a tensor product of differentials on $\Sigma\times\Sigma$, uniquely defined by the properties that it has a double pole on the diagonal of zero residue, double residue equal to $1$, no further singularities and normalised by $\int_{p\in{\cal A}_i}B(p,p')=0$, $i=1,...,g$.  It satisfies the Cauchy property for any meromorphic function $f$ on $\Sigma$
\beq  \label{cauchy}
df(p)=\Res_{p'=p}f(p')B(p,p').
\eeq
\end{definition}

On $\Sigma$, choose a Bergman kernel $B(p,p')$ normalised to have zero periods over the $\cal A$-cycles in the Torelli marking $\{ {\cal A}_i,{\cal B}_i\}$.  For any $\cc_{\alpha}\in\cd$ define a {\em primary differential} by 
\beq \label{primdif}
\phi_{\alpha}(p)=\oint_{p'\in\cc_{\alpha}^*}B(p,p')
\eeq
where $\cc_{\alpha}^*$ is a cycle dual to $\cc_{\alpha}$ defined in section 4. Each primary differential is locally holomorphic on $\Sigma\setminus x^{-1}(\infty)$.

%Note that $y_{\alpha}(p)$ is only a locally defined function on $\Sigma$ (well-defined on $\Sigma\setminus\{ {\cal A}_i,{\cal B}_i\}$).  

Denote by $\cp_i\in\Sigma$ the finite critical points of $x$, i.e. $dx(\cp_i)=0$.  For a generic point in $\widetilde{H}_{g,\mu}$ the critical points of $x$ are simple and the critical values $u_i=x(\cp_i)$, $i=1,...,N$ of $x$ are local coordinates in the open dense domain of $\widetilde{H}^s_{g,\mu}\subset\widetilde{H}_{g,\mu}$ defined by $u_i \neq u_j$ for $i \neq j$ and $\{\ca_1,...\ca_g,\cb_1,...,\cb_g\}$ avoid $x^{-1}(\infty)$.  

Define a metric on $\widetilde{H}^s_{g,\mu}$ by
\beq  \label{metric}
\eta=\sum_{i=1}^Ndu_i^2\cdot\Res_{\cp_i}\frac{\phi \cdot \phi}{dx}
\eeq
for any choice of primary differential $\phi=\phi_{\alpha}$ on $\Sigma$ obtained from $\cc_{\alpha}\in\cd$ via \eqref{primdif}.  
\begin{theorem}[Dubrovin \cite{Dub94}]
(i) The metric $\eta$ defined in \eqref{metric} is flat with local flat coordinates given by
\beq  \label{flat}
t_{\beta}=\int_{\cc_{\beta}}\phi,\quad \cc_{\beta}\in\cd,\quad\beta=1,...,N
\eeq
i.e. the metric is constant with respect to the coordinates $t_{\beta}$.  

(ii) The flat metric $\eta$ forms part of a Frobenius manifold structure on $\widetilde{H}_{g,\mu}$ with multiplication on the tangent space of $\widetilde{H}^s_{g,\mu}$ defined using the local basis of vector fields $\partial_{u_i}$ by
\beq  \label{canprod}
 \partial_{u_i}\cdot \partial_{u_j}= \delta_{ij} \partial_{u_i}.
\eeq
%hence the $\{u_i\}$ are canonical coordinates on $\widetilde{H}^s_{g,\mu}$.
\end{theorem}
The theorem was proven by Dubrovin in \cite{Dub94} using a definition of primary differential via deformations of $y_\alpha dx$---see Lemma~\ref{varprimdif}.   As stated here we use an equivalent definition of primary differential \eqref{primdif} proven by Shramchenko \cite{ShrRea}.\\

Recall that a Frobenius manifold $M$ comes equipped with a a flat metric $\eta$ together with a commutative, associative product $\cdot$ on its tangent space satisfying the compatibility condition $\eta(u\cdot v,w)=\eta(u,v\cdot w)$ for all $u,v,w\in T_pM$.  %The metric \eqref{metric} is produced from the product \eqref{canprod} together with the 1-form $\omega=\displaystyle\sum_{i=1}^Ndu_i\cdot\Res_{\cp_i}\frac{dy\cdot dy}{dx}$.
Associated to each semi-simple point $p$ of a Frobenius manifold $M$ is a cohomological field theory \cite{FSZTau,GiventalSem,ManFro,Teleman} defined on $(H,\eta)=(T_pM,\eta|_{T_pM})$, which is a sequence of $S_n$-equivariant maps 
\[ I_{g,n}:H^{\otimes n}\to H^*(\overline{\modm}_{g,n})\]
that satisfy gluing conditions on boundary divisors in $\overline{\modm}_{g,n}$ given explicitly in Section~\ref{sec:cohft}.  For any collection of vectors $v_1,...,v_n\in T_pM$, the integral $\int_{\overline{\modm}_{g,n}}I_{g,n}(v_1\otimes...\otimes v_n)\in\bc$ (which is a function of $p\in M$) is known as a primary invariant of $M$.

Recently, \cite{DOSS12} explained that one can compute the primary (and ancestor) invariants of a semisimple Frobenius manifold efficiently using the topological recursion procedure of \cite{EOInv}. The present review explains how this result can be applied to Hurwitz Frobenius manifolds.

The main observation of this text is  that just as the flat coordinates can be obtained as periods of a primary differential along cycles taken from $\cd$ via \eqref{flat}, the primary  invariants of the Hurwitz Frobenius manifolds can also be obtained as periods along cycles taken from $\cd$.  Since we need multiple insertions of vectors into the primary invariants, we need to take periods of symmetric {\em multidifferentials} on $\Sigma$ which are tensor products of differentials on $\Sigma^n=\Sigma\times...\times\Sigma$.  
\begin{theorem}  \label{th:main0}
Given a point $p=(\Sigma,x,\{ {\cal A}_i,{\cal B}_i\})\in\widetilde{H}^s_{g,\mu}$ and a choice of primary differential $dy_{\alpha}$ that determines a Frobenius manifold structure on $\widetilde{H}^s_{g,\mu}$, there exist multidifferentials $\omega_{g,n}$ defined on $\Sigma$ whose periods along contours in $\cd$ give the primary invariants of the Frobenius manifold at $p$.  More precisely, for flat coordinates $\{t_1,...,t_N\}$, put $e_{\alpha}=\partial_{t_{\alpha}}$ and define the dual vector with respect to the metric \eqref{metric} by $e^{\alpha}=\displaystyle\sum_{\beta}\eta^{\alpha\beta}e_{\beta}$.  Then
$$
\int_{\cc_{\alpha_1}}...\int_{\cc_{\alpha_n}}\omega_{g,n}=\int_{\overline{\modm}_{g,n}}I_{g,n}\Big(e^{\alpha_1}\otimes ...\otimes e^{\alpha_n}\Big)
$$
where $\cc_{\alpha_i}$ and $e_{\alpha_i}=\partial_{t_{\alpha_i}}$ are related by \eqref{flat}
\end{theorem}
Theorem~\ref{th:main0} is a consequence of the more general Theorem~\ref{th:main} that proves that the $\omega_{g,n}$ store all ancestor invariants of the Frobenius manifold, using a larger class of cycles than those in $\cd$.  
\begin{remark}
The statement and conclusion of Theorem~\ref{th:main0} can be made for any point in $\widetilde{H}_{g,\mu}$ not just the semisimple points $\widetilde{H}^s_{g,\mu}\subset\widetilde{H}_{g,\mu}$.  It would be interesting to prove the theorem with these weaker hypotheses.  There are candidate multidifferentials, such as those defined in \cite{BouchardEynard} where the zeros of $dx$ are not required to be simple, or in the case of Dubrovin's superpotential, studied from the perspective of topological recursion in \cite{DNOPS}, which applies to {\em any} semi-simple Frobenius manifold, and where there may be multiple zeros of $dx$ above a critical value.
\end{remark}

The multidifferentials $\omega_{g,n}$ in Theorem~\ref{th:main0} are obtained  from the topological recursion procedure associated to the spectral curve $(\Sigma,x,y_{\alpha},B)$ where $B=B(p,p')$ is the Bergman kernel defined in Definition~\ref{bergman} using the Torelli marking and $y_\alpha$ is a function defined on $\Sigma\setminus\{ {\cal A}_i,{\cal B}_i\}$ such that locally $\phi_\alpha = d y_\alpha$ . In general \cite{EOInv}, the $\omega_{g,n}$ are a family of symmetric multidifferentials on the spectral curve that encode solutions of a wide array of problems from mathematical physics, geometry and combinatorics.  By a spectral curve\footnote{The term spectral curve is inherited from the matrix model origin of this formalism. In the general formalism, this term is expected to make sense due to the probable existence of an associated integrable system.} we mean the data of $(\Sigma,x,y,B)$ given by a Riemann surface $\Sigma$ equipped with a meromorphic function $x$ and a locally defined meromorphic function $y\colon \Sigma\to \mathbb{C}$ such that the zeros of $dx$ given by $\{\cp_1,...,\cp_N\}$ are simple and $dy$ is analytic and non-vanishing on $\{\cp_1,...,\cp_N\}$, and equipped with a symmetric bidifferential $B$ on $\Sigma\times\Sigma$, with a double pole on the diagonal of zero residue, double residue equal to $1$, and no further singularities. The spectral curve may be a collection of $N$ open disks, known as a {\em local spectral curve}, because $\omega_{g,n}$ are defined using only local information about $x$, $y$ and $B$ around zeros of $dx$---see Section~\ref{sec:TR}.   On a compact spectral curve we relax the condition on $y$ being globally defined, and instead require that $dy$ is a locally defined meromorphic differential (a connection) ambiguous up to $dy+df(x)$ for any rational function $f$.  This gives rise to a locally defined function $y$ on $\Sigma$ which is sufficient to apply topological recursion.

In \cite{DOSS12,MilanovAncestor}, it was proven that, starting from a semi-simple  CohFT, or equivalently a semi-simple Frobenius manifold $M$, it is possible to compute its correlation functions by the topological recursion procedure applied to a specific local spectral curve:
\beq \label{DOSS}
\{\text{semisimple CohFT}\}\longrightarrow\{\text{topological recursion applied to a local spectral curve}\}
\eeq
Under this correspondence whose details are reviewed in Section~\ref{sec:DOSS}, the number of zeros of $dx$ on the local spectral curve $(\Sigma,x,y,B)$ is equal to the dimension $N$ of the Frobenius manifold $M$.
  It was then proven in \cite{DNOPS} that, under some additional assumptions on the Frobenius manifold $M$, it is possible to arrange that the image of \eqref{DOSS} is a {\em compact} spectral curve producing the same correlation functions.  This compact Riemann surface is given by Dubrovin's superpotential \cite{Dub94,Dub98} which is a family of compact Riemann surfaces parametrised by the semi-simple points of $M$ and constructed out of flat coordinates of a pencil of metrics on $M$.%, or equivalently the period map for monodromy data from the Riemann-Hilbert problem associated to $M$.

We can now try to reverse the direction of the arrow in \eqref{DOSS}.  Given a {\it compact} spectral curve, when does it lie in the image of \eqref{DOSS}, and can we reconstruct the corresponding CohFT (or, equivalently, the Frobenius manifold)?   The following theorem answers this question.  It begins with the observation that a compact spectral curve gives a point $(\Sigma,x)$ in a Hurwitz space $H_{g,\mu}$.

%This definition of a primary differential $dy$ does not require a family of spectral curves.  

\begin{theorem}  \label{th:main}
Given a generic point $(\Sigma,x)\in H_{g,\mu}$, equip it with a bidifferential $B$ normalised over a given set of %\textcolor{red}{geometric} 
$\cal A$-cycles, and choose $\cc\in\cd$ to define a locally defined function $y$ on $\Sigma$ by $dy(p):=\oint_\cc B(p,p')$.  The topological recursion procedure applied to the spectral curve $(\Sigma,x,y,B)$ computes the ancestor invariants of the CohFT associated to Dubrovin's Frobenius manifold structure on the cover $\widetilde{H}^s_{g,\mu}$ via:
\beq  \label{TRCohFT}
\omega_{g,n}(p_1,...,p_n)=\sum_{\substack{i_1,\dots,i_n \\ d_1,\dots,d_n}} \int_{\overline{\mathcal{M}}_{g,n}} I_{g,n}(e_{i_1},\dots,e_{i_n}) \prod_{j=1}^n
\psi_j^{d_j} \cdot \bigotimes_{j=1}^n V^{i_j}_{d_j}(p_j)
\eeq
where $V^i_k(p)$, $i=1,...,N$, $k=0,1,...$, are canonical differentials on $\Sigma$ defined by \eqref{Vdiff} in Section~\ref{sec:DOSS}.
\end{theorem}
The proof of Theorem~\ref{th:main}---contained in Section~\ref{sec:rauch}---is a simple combination of results from \cite{Dub98,DOSS12,GiventalMain,ShrRie} which are reviewed in this paper.  The main tool in the proof is the map \eqref{DOSS} from \cite{DOSS12} which shows how topological recursion relates to Givental's construction \cite{GiventalMain} of the total ancestor potential associated to each semi-simple point of a Frobenius manifold.   To apply the reverse construction of \eqref{DOSS} one needs a specific relationship between the Bergman kernel $B$ on the spectral curve and the $R$-matrix of the Frobenius manifold which is proven in \cite{ShrRie}.  %The image of \eqref{DOSS} consists of local spectral curves with spectral curve data $dy$ and $B$ satisfying a compatibility condition \eqref{DOSStest} together with the condition that the Bergman kernel factorises into a special form---see \eqref{shape2} in Section~\ref{sec-factor}. The first main obstacle is to characterise compact spectral curves that lie in the image of \eqref{DOSS} via sufficient conditions to satisfy \eqref{DOSS} and \eqref{shape2}.  Compactness of the curve is a very strong requirement---which allows us to show that both conditions are satisfied when $dy$ is a primary differential.  Once we have determined that there is indeed a CohFT associated to the compact spectral curve the second main obstacle is to identify the CohFT, specifically to show that it coincides with the Hurwitz Frobenius manifold. 
\begin{remark}
Theorem~\ref{th:main} also answers the following question. Given a compact spectral curve, what does the topological recursion procedure compute? For a large class of spectral curves---where $B$ and $y$ are determined almost canonically by $\Sigma$ and $x$---the answer is that it produces generating functions for ancestor invariants of a Hurwitz space to which the branched cover underlying the spectral curve belongs. In particular, it completes the picture drawn by Zhou in \cite{ZhoFro} for relating Frobenius manifolds and spectral curves. 
\end{remark}
\begin{remark}
Theorem~\ref{th:main0} concerns only the primary invariants in \eqref{TRCohFT}, corresponding to $d_j=0$, $j=1,...,n$.   One can also construct generalised contours $\cc_{\alpha,k}=p_k(x)\cc_{\alpha}$, for $\cc_{\alpha}\in\cd$ and $p_k(x)=x^k+...$ a monic polynomial of degree $k$ in $x$, so that the ancestor invariants, corresponding to $d_j\geq 0$, appear as periods thus generalising Theorem~\ref{th:main0}:
\beq
\int_{\cc_{\alpha_1,k_1}}...\int_{\cc_{\alpha_n,k_n}}\omega_{g,n}=\int_{\overline{\modm}_{g,n}}I_{g,n}\Big(e^{\alpha_1}\otimes ...\otimes e^{\alpha_n}\Big)\cdot\prod_{j=1}^n\psi_j^{k_j}.%\langle\prod_{j=1}^n\tau_{k_j}\Big(e^{\alpha_j}\Big)\rangle.
\eeq
\end{remark}

Theorems~\ref{th:main0} and ~\ref{th:main} enable one to generate primary invariants and all ancestor invariants of $\widetilde{H}^s_{g,\mu}$ of all genera.  Previously only genus 0 and genus 1 primary invariants were known.  Theorems~\ref{th:main0} and ~\ref{th:main} also have applications to the topological recursion procedure.  Using the generalised contours in $\cd$ one gets a direct map from $\omega_{g,n}$ to primary invariants via integration over the cycles.

%We also describe a dictionary between properties of the Frobenius manifold and properties of the spectral curve, such as Givental's $R$-matrix, the transition matrix from the flat to the normalised canonical bases, homogeneity, the string equation and reconstruction from the data of a topological field theory. This puts many of the properties of the correlation functions, such as variational formulae and residue properties, described in \cite{EOInv} in the context of Frobenius manifolds. 

%A more general result is given in Theorem~\ref{::} where we essentially take any compact spectral curve in the image of \eqref{DOSS}.  This includes deformed examples of the Frobenius manifold structures of Dubrovin by Shramchenko \cite{ShrDef}.  More precisely, we say that a locally defined meromorphic differential $dy$ on $\Sigma$ is {\em admissible} for $dx$ if its zeros avoid the zeros of $dx$ and its poles are dominated by the poles of $dx$.  In particular, a $dy$ that is admissible for $dx$ lies in the image of \eqref{DOSS} for any choice of $B$.  Note that an admissible $dy$ can now be ambiguous only up to $dy+\lambda dx$ for $\lambda\in\bc$ to preserve this condition.  Theorem~\ref{::} applies to any given compact $(\Sigma,x)$ equipped with an admissible $dy$ and any globally defined $B$.  Unlike above, $y$ and $B$ are not necessarily normalised over the same cycles.  One again obtains Frobenius manifold structures on the cover of a Hurwitz space although possibly with the loss of the homogeneity property.

The paper is organised as follows.  In Sections~\ref{sec:fro} and \ref{sec:TR}, we remind the reader of the general theory of Frobenius manifolds and topological recursion, as well as the correspondence between the two, following \cite{DOSS12}.  In Section~\ref{sec:hur}, we describe the construction of Dubrovin of Frobenius manifold structures on covers of Hurwitz spaces and prove Theorems~\ref{th:main0} and \ref{th:main}.  In Section~\ref{sec:compact}, 
we discuss an extension of the results to non-semi-simple points.
%we express the metric and product of this Frobenius manifold in terms of the initial data of the topological recursion. We identify the correlators of this theory with periods of the correlation functions of the topological recursion by identifying the 3-point function of the topological recursion with the product of the Frobenius structure.

%\textcolor{red}{Could it be that expansion of $\omega_{g,n}$ in a local coordinate around $\infty_i$ defines some type of descendants?  Or, at least, can this be represented by an action of some type of $S$-matrix? (Then naturally there is a choice for each $\infty_i$.)}

\subsection{Acknowledgement}
The authors would like to thank Gaetan Borot for fruitful discussions on the subject of this paper.   N.O. would like to thank the KdV-Institute where part of this work was done.  P.N. would like to thank Ludwig-Maximillians-Universit\'at, Munich and the Max-Planck-Insititut f\"ur Mathematik, Bonn for hosting him while part of this work was carried out.  A.P. and S.S. were supported by the Netherlands Organisation for Scientific Research. 
P.D.-B. was partially supported by RFBR grants 16-31-60044-mol-a-dk and
15-01-05990, and by joint RFBR-India grant 16-51-45029-Ind. A.P. was
partially supported by RFBR grant 16-01-00291. P.D.-B. and A.P. were also partially supported
by RFBR grant 15-31-20832-mol-a-ved and by joint RFBR-Japan grant
15-52-50041-YaF.

\section{Frobenius manifolds} \label{sec:fro}
In this section we give a short introduction to Frobenius manifolds.  An important construction for this paper is Givental's $R$-matrix defined in Section~\ref{sec:Rmatrix}.
\subsection{Frobenius manifold}
\begin{definition}
A Frobenius algebra $(H,\eta,\cdot)$ is a finite-dimensional vector space $H$ equipped with a metric $\eta=\langle\ ,\ \rangle$ and a commutative, associative product $\cdot$ satisfying $\langle u\cdot v,w \rangle=\langle u, v\cdot w \rangle$.  %The metric can be defined in terms of a linear functional $\theta$ on $H$ by $\theta(u\cdot v)=$\langle u,v \rangle$.
\end{definition}
\begin{example}  \label{semsim}
$$H\cong\bc\oplus\bc\oplus...\oplus\bc,\quad\langle e_i,e_j \rangle=\delta_{ij}\eta_i,\quad e_i\cdot e_j=\delta_{ij}e_i$$ 
for any $\eta_i\in\bc \setminus \{0\}$, $i=1,...,N$ and where $\{e_i\}$ is the standard basis.  Conversely any {\em semisimple} Frobenius algebra is determined uniquely by $N$ non-zero complex numbers $\{\eta_i\}$ and is isomorphic to this example. 
\end{example}

A Frobenius manifold is defined by the data of a Frobenius algebra on the tangent space at each point of the manifold and such that the metric is flat.  In terms of flat coordinates $\{t^{\alpha}\}$ a Frobenius manifold can be defined locally as follows.  Consider a function $F(t^1,\dots,t^N)$ defined on a ball $B\subset \mathbb{C}^N$ and a constant inner product $\eta^{\alpha\beta}$ such that the triple derivatives of $F$ with one raised index,
\begin{equation}
C_{\alpha\beta}^\gamma:=\frac{\partial^3 F}{\partial t^\alpha \partial t^\beta \partial t^\lambda} \eta^{\lambda\gamma},
\end{equation}
are the structure constants of a commutative associative Frobenius algebra with the scalar produce given by $\eta_{\alpha\beta}$. We can think about this structure as defined on the tangent bundle of $B\subset \mathbb{C}^N$ (and we denote the corresponding multiplication of vector field by $\cdot$), and we require that $\partial_{t^1}$ is the unit of the algebra in each fibre. 

We further consider structures (almost) homogeneous under a vector field $E:=\sum_{\alpha=1}^N ((1-q_\alpha)t^\alpha+r_\alpha) \partial_{t^\alpha}$, where $q_\alpha$ and $r_\alpha$ are constants for $\alpha=1,\dots,N$, satisfying $q_1=0$ and $r_\alpha\not=0$ only in the case $1-q_\alpha=0$. We require that there exists a constant $d$ such that $E.F-(3-d)F$ is a polynomial of order at most $2$ in $t^1,\dots,t^N$. 

The triple $(F,\eta,E)$ that satisfies all conditions above gives us the structure of a (conformal) Frobenius manifold of rank $N$ and conformal dimension $d$ with flat unit. The function $F$ is called the prepotential; the vector field $E$ is called the Euler vector field.  The coordinate-free description of this structure requires a flat metric with associated Levi-Civita connection, unit and Euler vector fields satisfying compatibility conditions---see \cite{Dub94} for details.

In this paper we only consider semi-simple Frobenius manifolds, that is, we require that the algebra structure at each point on an open subset $B^{ss}\subset B$ is semi-simple hence isomorphic to Example~\ref{semsim}. In a neighborhood of a semi-simple point we have a system of canonical coordinates $u_1,\dots,u_N$, defined up to permutations, such that the vector fields $\partial_{u_i}$, $i=1,\dots,N$, are the idempotents of the algebra product, and the Euler vector field has the form $E=\sum_{i=1}^N u_i\partial_{u_i}$.  This gives rise to two important systems of coordinates: flat coordinates, leading to a fixed metric and varying product, and canonical coordinates, leading to a fixed product and varying metric.
With respect to the canonical coordinates, the flat metric on $M$ is diagonal with diagonal terms generated by a potential $H:M\to\bc$ via \eqref{potmet} which satisfies \eqref{diagmet}.

Define the rotation coefficients
\begin{equation}  \label{rotcoeff}
\beta_{ij}=\frac{\partial_{u_j}\eta_i}{\sqrt{\eta_i\eta_j}} .
\end{equation}
Then \eqref{diagmet} and \eqref{potmet} imply that $\beta_{ij}(u)$ satisfy the Darboux-Egoroff system
\begin{align}
\beta_{ij}&=\beta_{ji} , \\
\label{rotcoeffPDE} 
\partial_{u_k}\beta_{ij}&=\beta_{ik}\beta_{jk}.
\end{align}
Flatness of the identity and conformality imply
\begin{align}
\sum_{k}\partial_{u_k}\beta_{ij}&=0,\\
\label{rotcoeffEul} 
\sum_{k}u_k\partial_{u_k}\beta_{ij}&=-\beta_{ij}.
\end{align}
Assemble the rotation coefficients into a symmetric $N\times N$ matrix $\Gamma=\Gamma(u)$---whose diagonal is not a part of the structure---by $\Gamma_{ij}=\beta_{ij}$.  Then equations (\ref{rotcoeffPDE}-\ref{rotcoeffEul} ) are equivalent to the Darboux-Egoroff equation
\begin{equation}
d[\Gamma,U]=[[\Gamma,U],[\Gamma,dU]]
\end{equation}
where $U=\diag(u_1,\dots,u_N)$. 

The rotation coefficients give less information than the metric, i.e. there are different solutions of \eqref{diagmet} and \eqref{potmet} that give rise to the same rotation coefficients.  The system
\begin{align*}
\partial_{u_j}\psi_i&=\beta_{ij}\psi_j,\quad i\neq j\\
\sum_{j=1}^N\partial_{u_j}\psi_i&=0,\quad i=1,...,N
\end{align*}
has an $N$-dimensional space of solutions $\psi=(\psi_1(u),...,\psi_N(u))$ which enables one to retrieve a metric for each solution from the rotation coefficients.  Put $N$ independent solutions of this system into the columns of a matrix $\Psi$, so the system becomes
\beq \label{psipde}
d\Psi = [\Gamma,dU] \Psi.
\eeq
The canonical coordinate vector fields $\partial_{u_i}$ are orthogonal but not orthonormal. We can normalise them to produce a so-called normalised canonical frame in each tangent space, that is, if $\eta_i=\eta(\partial_{u_i},\partial_{u_i})$, then the orthonormal basis is given by $\partial_{v_i}:=\eta_i^{-1/2} \partial_{u_i}$, $i=1,\dots,N$. The matrix $\Psi$ in \eqref{psipde} is the transition matrix from the flat basis to the normalised canonical basis.

\subsection{Cohomological field theory}  \label{sec:cohft}

A {\em cohomological field theory} is a pair $(H,\eta)$ composed of a finite-dimensional complex vector space $H$ equipped with a metric $\eta$ and a sequence of $S_n$-equivariant maps. 
\[ I_{g,n}:H^{\otimes n}\to H^*(\overline{\modm}_{g,n})\]
that satisfy compatibility conditions from inclusion of strata:
%\[ \overline{\modm}_{g_1,n_1+1}\times\overline{\modm}_{g_2,n_2+1}\to\overline{\modm}_{g_1+g_2,n_1+n_2}\]
%$\displaystyle I_{g,n}(\alpha_1\otimes\cdots\otimes\alpha_n)=I_{g_1,n_1+1}\otimes I_{g_2,n_2+1}\big(\bigotimes_{j\in S_1}\alpha_j\otimes\Delta\otimes \bigotimes_{j\in S_2}\big)$ where $\Delta\in H\otimes H$ is dual to $\eta$.  
$$\psi:\overline{\modm}_{g-1,n+2}\to\overline{\modm}_{g,n},\quad \phi_I:\overline{\modm}_{g_1,|I|+1}\times\overline{\modm}_{g_2,|J|+1}\to\overline{\modm}_{g,n},\quad I\sqcup J=\{1,...,n\}$$
%then for $\displaystyle\Delta=\sum_{\alpha,\beta}\eta^{\alpha\beta}e_{\alpha}\otimes e_{\beta}$
given by
\begin{align}
\phi_I^*I_{g,n}(v_1\otimes...\otimes v_n)&=I_{g_1,|I|+1}\otimes I_{g_2,|J|+1}\big(\bigotimes_{i\in I}v_i\otimes\Delta\otimes\bigotimes_{j\in J}v_j\big)  \label{glue1}
\\
\psi^*I_{g,n}(v_1\otimes...\otimes v_n)&=I_{g-1,n+2}(v_1\otimes...\otimes v_n\otimes\Delta)  \label{glue2}
\end{align}
where $\Delta\in H\otimes H$ is dual to the metric.  In local coordinates it is given by $\Delta=\eta^{\alpha\beta}e_{\alpha}\otimes e_{\beta}$.

The metric $\eta=\langle\cdot,\cdot\rangle$ and the 3-point function $I_{0,3}$ induce a product $\cdot$ on $H$ via
$$ \langle u\cdot v,w\rangle=I_{0,3}(u,v,w)\in H^*(\overline{\modm}_{0,3})\cong\bc.
$$

Correlators, or ancestor invariants, of the CohFT make use of the Chern classes $\psi_j=c_1(L_j)$ of the tautological line bundles $L_j,j=1,...,n$ over $\overline{\modm}_{g,n}$.  The correlators are defined by: 
\beq \label{tau} 
\langle\tau_{k_1}(e_{\nu_1})...\tau_{k_n}(e_{\nu_n})\rangle_g:=\int_{\overline{\modm}_{g,n}}I_{g,n}(e_{\nu_1},...,e_{\nu_n})\cdot\prod_{j=1}^n\psi_j^{k_j}
\eeq
for $k_i\in\bn$, $\{e_{\nu},_{\ \nu=1,...,N}\}\subset H$.  When $k_i=0$, $i=1,...,n$ the ancestor invariants are also known as {\em primary invariants} of the CohFT.

Givental \cite{GiventalMain} introduced a group action on genus 0 potentials of a CohFT, and used it to propose a formula for higher genera.  Faber, Zvonkine and the last author \cite{FSZTau} proved that the higher genera formula satisfies all properties that might be imposed to correlators of CohFT, hence the Givental group acts on partition functions of CohFTs in all genera.  The interpretation of the action on correlators as an action on cohomology classes was constructed by several people independently, namely, by Teleman \cite{Teleman}, Katzarkov-Kontsevich-Pantev (unpublished), and Kazarian (unpublished)---see \cite{ShaBCO}.  There is a good account of this action on cohomology classes by Pandharipande-Pixton-Zvonkine \cite{PPZRel}.  Hence we can associate a CohFT to a semi-simple point of a Frobenius manifold.  Conversely a CohFT gives rise to a Frobenius manifold structure on (a neighborhood inside) $H$ using the constant metric $\eta$ as the flat metric and a varying family of products using $I_{0,n}$ in place of $I_{0,3}$.  See \cite{ManFro} for details.

If the Frobenius manifold has flat identity---meaning that the identity vector field for the product on the tangent bundle is parallel with respect to the Levi-Civita connection of the flat metric $\eta$---then this is realised on the CohFT level by an extra relation involving the forgetful map
$$\pi:\overline{\modm}_{g,n+1}\to\overline{\modm}_{g,n}$$
given by
\beq  \label{flatid}
I_{g,n+1}(v_1\otimes\cdots\otimes v_n\otimes \un)=\pi^*I_{g,n}(v_1\otimes\cdots\otimes v_n),\quad I_{0,3}(v_1\otimes v_2\otimes \un)=\eta(v_1\otimes v_2)
\eeq
where $\un$ is the unit vector for the product.
%{\em Universal behaviour} $\displaystyle F^{\text KW}=\sum\langle\tau_0^{k_0}\tau_1^{k_1}...\rangle_g\frac{t_0^{k_0}}{k_0!}\frac{t_1^{k_1}}{k_1!}...$

\subsection{Classification of semi-simple cohomological field theories}
\label{sec:Rmatrix}

The Givental-Teleman theorem \cite{GiventalMain,Teleman} states that a semi-simple CohFT is equivalent to the pair $(H,\eta)$ together with a so-called $R$-matrix.  An $R$-matrix 
$$R(z)  =  \sum_{k=0}^\infty R_k z^k$$
is a formal series whose coefficients are $N \times N$ matrices where $N=\dim H$ is the rank of the Frobenius manifold.  Givental used $R[z]$ to produce a differential operator, a so-called quantisation of $R[z]$, which acts on a known tau-function to produce a generating series for the correlators of the CohFT.

The coefficients $R_k$ are defined using $\Psi$, the transition matrix from flat coordinates to normalised canonical coordinates determined by \eqref{psipde}, via $R_0=I$ and the inductive equation
\beq\label{eqdefR}
d\left( R(z) \Psi \right) = \frac{\left[R(z),dU\right]}{z} \Psi
\eeq
which uniquely determines $R(z)$ up to left multiplication by a diagonal matrix $D(z)$ independent of $u$ with $D(0)=I$. We recall that this  equation is a consequence of the fact that $R(z)$ is the regular part of the expansion of the solution of a linear system associated by Dubrovin  to any semisimple Frobenius manifold around its essential singularity (see for example lecture 3 in \cite{Dub94} for more details).

%\textcolor{red}{I changed the conventions compared to Givental's ones which were very confusing. Here, we have the usual multiplication rule for matrices. It seems to agree with the conventions for $\Gamma$. Let me know if you agree with these conventions and the equations as written here.}
Using $d\Psi = \left[\Gamma,dU\right] \Psi$ one can write
$$
d\left(R(z) \Psi \right) = d\left[R(z)\right] \Psi + R(z) d \Psi = 
d\left[R(z)\right] \Psi + R(z)  \left[\Gamma,dU\right] \Psi .
$$
Together with equation \eqref{eqdefR} and the invertibility of $\Psi$, this gives
\beq  \label{eqdefR2}
dR(z) = \frac{\left[R(z),dU\right] }{ z} - R(z)  \left[\Gamma,dU\right].
\eeq
This re-expresses the equation for $R(z)$ in terms of the rotation coefficients, which uses less information than the full metric, encoded in $\Psi$.   Since $dU(\un) =I$, an immediate consequence of \eqref{eqdefR2} is
\beq \label{unR}
\un\cdot R(z)=0.
\eeq

If the theory is homogenous, then invariance under the action of the Euler field 
\beq \label{homR}
(z\partial_z+E)\cdot R(z)=0
\eeq
fixes the diagonal ambiguity in $R(z)$.

The first non-trivial term $R_1$ of $R(z)$ is given by the rotation coefficients
\beq \label{R1=Gam}
R_1=\Gamma.
\eeq
This follows from comparing the constant (in $z$) term in \eqref{eqdefR} which is
$$d\Psi=[R_1,dU]\Psi
$$
to equation \eqref{psipde} given by $d\Psi = \left[\Gamma,dU\right] \Psi$.  Since $dU$ is diagonal with distinct diagonal terms, we see that \eqref{R1=Gam} holds for off-diagonal terms, and the ambiguity in the diagonal term for both is unimportant---it can be fixed in $R_1$ by \eqref{eqdefR} together with $E\cdot R_1=-R_1$.

%In particular, the first equation in the hierarchy \ref{eqdefR} reads
%\beq\label{eqdefR1}
%{\partial \Psi_{i\alpha} \over \partial u_k} = \delta_{i,k} \sum_{j} \left[R_1\right]_i^j \Psi_{j\alpha}  - \left[R_1\right]_i^k\Psi_{k\alpha}
%\eeq
%fixing the relation between $\Psi$ and $R_1$. 

\section{Topological recursion and cohomological field theory} \label{sec:TR} 
In this section, we give a brief overview of topological recursion defined in \cite{EOInv}.
Consider a Riemann surface $\Sigma$ equipped with meromorphic functions $x,y\colon \Sigma\to \mathbb{C}$ such that the zeros of $dx$, given by $\{\cp_1,...,\cp_N\}$ are simple and $dy$ is analytic and non-vanishing on $\{\cp_1,...,\cp_N\}$.  Let $B$ be a Bergman kernel on $\Sigma\times\Sigma$ as in definition \ref{bergman}.

Define a sequence of symmetric multidifferentials $\omega_{g,n}(p_1,\dots,p_n)$ on $\Sigma^{\times n}$ by the following recursion:
\begin{align}
& \omega_{0,1}(p):=y(p) dx(p); \\
& \omega_{0,2}(p_1,p_2):=B(p_1,p_2);\\ \label{eq:TopologicalRecursion-1stTime}
& \omega_{g,m+1}(p_0,p_1,\dots,p_n) :=  \\ \notag
& \sum_{i=1}^N \Res_{p=\cp_i} 
\frac{\int_p^{\sigma_i(p)}\omega_{0,2}(\bullet,p_0)}{2(\omega_{0,1}(\sigma_i(p))-\omega_{0,1}(p))}
\tilde\omega_{g,2|n}(p,\sigma_i(p)|p_1,\dots,p_n),\label{TR}
\end{align}
where $\sigma_i$ is the local involution defined by $x$ near the point $\cp_i$, $i=1,\dots,N$, and $\tilde\omega_{g,2|n}$ is defined by the following formula:
\begin{align}
\tilde\omega_{g,2|n}(p',p''|p_1,\dots,p_n):=& \omega_{g-1,n+2}(p',p'',p_1,\dots,p_n)+
\\ \notag &
\sum_{\substack{g_1+g_2=g\\ I_1\sqcup I_2 = \{1,\dots,n\} \\ 2g_1-1+|I_1|\geq 0 \\ 2g_2-1+|I_2|\geq 0}}
\omega_{g_1,|I_1|+1}(p',p_{I_1})\omega_{g_2,|I_2|+1}(p'',p_{I_2}).
\end{align}
Here we denote by $p_I$ the sequence $p_{i_1},\dots,p_{i_{|I|}}$ for $I=\{i_1,\dots,i_{|I|}\}$.

\begin{remark}
The recursion was defined on so-called {\em local} spectral curves in \cite{EynardInt} as follows.  Consider small neighborhoods $U_i\subset \Sigma$ of the points $\cp_i$. If we look at just the restrictions of $\omega_{g,n}$ to the products of these disks, $U_{i_1}\times \cdots\times U_{i_n}$, we can still proceed by topological recursion, using as an input the restrictions of $\omega_{0,1}$ to $U_i$, $i=1,\dots,N$, and $\omega_{0,2}$ to $U_i\times U_j$, $i,j=1,\dots,N$. Indeed, Equation~\eqref{eq:TopologicalRecursion-1stTime} uses only local expansion data around the points $\cp_i$. Hence, the word local refers to the unique knowledge of these local data. 
\end{remark}

\begin{remark}  \label{y=prim}
In the topological recursion on a compact spectral curve we also allow $y$ to be the (multivalued) primitive of a differential $\omega$ on $\Sigma$.   The ambiguity in $y$ consists of periods and residues of $\omega$ and hence the ambiguity is locally constant.  Since $y$ appears in the recursion formula \eqref{TR} only via $y(\sigma_i(p))-y(p)$ (and there are no poles of $\omega$ at the zeros of $dx$) the locally constant ambiguity disappears and the recursion is well-defined.  We go even further and allow $\omega$ to be a locally defined meromorphic differential (a connection) ambiguous up to $dy+df(x)$ for any rational function $f$.  In this case the ambiguity $y\mapsto y+f(x)$ is no longer constant, but again $y(\sigma_i(p))-y(p)$ is unchanged.
\end{remark}

\subsection{Topological recursion from CohFTs} \label{sec:DOSS}
We recall the relation \eqref{DOSS} of topological recursion on a local spectral curve to the Givental formulae for cohomological field theories obtained in~\cite{DOSS12}. 

\begin{definition}\label{evaluationform}
For a Riemann surface equipped with a meromorphic function $(\Sigma,x)$ we define evaluation of any meromorphic differential $\omega$ at a simple zero $\cp$ of $dx$ by
$$
\omega(\cp):=\Res_{p=\cp}\frac{\omega(p)}{\sqrt{2(x(p)-x(\cp))}}%=\left.\frac{\omega(p)}{ds}\right|_{s=0}.
$$
where we choose a branch of $\sqrt{x(p)-x(\cp)}$ once and for all at each $\cp$ to remove the $\pm1$ ambiguity.
\end{definition}

\begin{theorem} \label{th:DOSS}  \cite{DOSS12}
Given a semi-simple CohFT presented via the $R$-matrix $R(z)  =  \sum_{k=0}^\infty R_k z^k$ and constants $\eta_1,...,\eta_N$ define a local spectral curve $(\Sigma,x,y,B)$, presented as (the Laplace transform of) local series for $dy(p)$ and $B(p,p')$ around each zero $p=\cp_i$, $p'=\cp_j$ of $dx$ (which is locally canonical) %Explicitly, the identification with the data of a CohFT then goes 
as follows: %\textcolor{red}{To be consistent with the following, I did write everything in terms of $R^{-1}(z)$ instead of  $R^{-1}(z^{-1})$}
\begin{align}
\label{eq:Identification-MatrixR}
& \left[R^{-1}(z)\right]^i_j = -\frac{\sqrt{z}}{\sqrt{2\pi}}\int_{\Gamma_j} B(\cp_i,p)\cdot e^{\frac{(u_j-x(p))}{ z}}\\
\label{eq:Identification-Dilaton}
& \sum_{k=1}^N \left[R^{-1}(z)\right]^k_i \cdot\eta_k^{1/2} = \frac{1}{\sqrt{2\pi z}}\int_{\Gamma_i}dy(p)\cdot e^{\frac{(u_i-x(p)) }{ z}}
\end{align}
\beq  \label{shape}
 \frac{1}{{2\pi \sqrt{z_1 z_2} }}\int_{\Gamma_i} \int_{\Gamma_j} {B(p_1,p_2)} e^{\frac{(u_i-x(p_1)) }{ z_1}+\frac{(u_j-x(p_2)) }{ z_2}}
 =
- \frac{\sum_{k=1}^N \left[R^{-1}(z_1)\right]^k_i
 	\left[R^{-1}(z_2)\right]^k_j}{z_1+z_2} 
\eeq
where $\Gamma_i$ is a path containing $u_i=x(\cp_i)$.  Then the multidifferentials $\omega_{g,n}(p_1,...,p_n)$ obtained via topological recursion applied to the local spectral curve $(\Sigma,x,y,B)$ are polynomials in  differentials $V^i_k(p_j)$ defined by
\beq  \label{Vdiff}
V^i_0(p)=B(\cp_i,p),\quad V^i_{k+1}(p)=d\left(\frac{V^i_k(p)}{dx(p)}\right),\ k=0,1,2,...
\eeq
with coefficients given by ancestor invariants of the CohFT:
$$\omega_{g,n}(p_1,...,p_n)=\sum_{\substack{i_1,\dots,i_n \\ d_1,\dots,d_n}} \int_{\overline{\mathcal{M}}_{g,n}} I_{g,n}(e_{i_1},\dots,e_{i_n}) \prod_{j=1}^n
\psi_j^{d_j} \cdot \bigotimes_{j=1}^n V^{i_j}_{d_j}(p_j).
$$
\end{theorem}  
\begin{remark}The spectral curve thus obtained is local, i.e. a collection of open sets $U_i$ each containing a unique zero $\cp_i$ of $dx$.  Thus $\Gamma_i$ is defined only locally, which is fine since we are interested only in the asymptotic expansion for $R$ around $z=0$. Let us also remind the reader that this result is valid for any semisimple Frobenius manifold. We shall see in the next section that, in the case of Hurwitz Frobenius manifolds, one can make these Laplace transform globally well-defined by choosing carefully the integration cycles to consider.
\end{remark}
\begin{remark}This data (the constants $\eta_i$ and the matrix $R(z)^j_i$) determine for us a semi-simple CohFT $\{I_{g,n}\}$ with an $N$-dimensional space of primary fields $V:=\langle e_1,\dots,e_N \rangle$ corresponding to a chosen point $(u_1,...,u_N)$ on a Frobenius manifold---see Section~\ref{sec:Rmatrix}.  In terms of the underlying Frobenius manifold structure, the basis $e_1,\dots,e_N$ corresponds to the normalised canonical basis
\end{remark}
 %\textcolor{red}{We should better define the Lefschetz thimbles as classes in a relative homology. Don't you think?} 
\begin{remark}Note that the limit of \eqref{eq:Identification-Dilaton} at $z=0$ yields:
\beq  \label{dyeta}
\eta_i^{1/2} = dy(\cp_i)
\eeq
which tells us that $dy$ encodes the metric.
\end{remark}
\begin{remark}\label{rem:shape}
Compatibility of \eqref{eq:Identification-MatrixR} and \eqref{shape} is a condition on the bidifferential $B$, not satisfied in general, nevertheless always satisfied if the spectral curve is compact and the differential $dx$ is meromorphic.   Compatibility for compact spectral curves uses a general finite decomposition for $B(p_1,p_2)$ proven by Eynard in Appendix B of \cite{EynardInv} together with \eqref{eq:Identification-MatrixR}. This is recalled in section \ref{sec-factor}.
%We assume Equation~\eqref{eq:decomposition-doubleLaplace-B} throughout the paper.
\end{remark}
\noindent Theorem~\ref{th:DOSS} produces a map 
$$\{\text{semisimple CohFT}\}\longrightarrow\{\text{topological recursion applied to a local spectral curve}\}$$
with image consisting of spectral curves with $B$ and $y$ necessarily satisfying compatibility conditions---compatibility of \eqref{eq:Identification-MatrixR}, \eqref{eq:Identification-Dilaton} and \eqref{shape}.  A general spectral curve will not satisfy such compatibility conditions, i.e. in general one can choose $B$ and $y$ independently.  For example, the rational spectral curve $(\bp^1,x,y,B)$ for $x=z+1/z$, $B=dzdz'/(z-z')^2$, $dy=z^mdz$, $m\in\{-1,0,1,2,...\}$, lies in the image of the map only for $m=-1$ or $0$.

Compatibility of \eqref{eq:Identification-MatrixR} and \eqref{shape} is discussed in Remark~\ref{rem:shape} and compatibility of \eqref{eq:Identification-MatrixR} and~\eqref{eq:Identification-Dilaton} is characterised by the following theorem. 
\begin{theorem}[\cite{DNOPS}]
	\label{thm:CompatibilityTest} Equations~\eqref{eq:Identification-MatrixR} and~\eqref{eq:Identification-Dilaton} are compatible (as equations for the unknown variables $R^{-1}$ and $\eta_i$), $i=1,\dots,N$) if and only if the 1-form
\beq  \label{DOSStest}
\omega(p)=d\left(\frac{dy}{dx}(p)\right)+\sum_{i=1}^N \Res_{p'=\cp_i}\frac{dy}{dx}(p')B(p,p').
\eeq
is invariant under each local involution $\sigma_i$, $i=1,\dots,N$.  
\end{theorem}
The characterisation in Remark~\ref{rem:shape} and Theorem~\ref{thm:CompatibilityTest} allows a {\em converse} construction of semisimple CohFTs from compact spectral curves.  The following is a sufficient condition for compatibility of \eqref{eq:Identification-MatrixR} and~\eqref{eq:Identification-Dilaton}.
\begin{definition}  \label{def:dominant}
We say that a compact spectral curve $(\Sigma,x,y,B)$ is {\em dominant} if $x$ and $dy$ are meromorphic and the poles of $dx$ dominate the poles of $dy$.
\end{definition}
\begin{corollary}  \label{spec2cohft}
A dominant compact spectral curve $(\Sigma,x,y,B)$ lies in the image of \eqref{DOSS} and hence gives rise to a semisimple CohFT.
\end{corollary}
\begin{proof}
Any Bergman kernel satisfies the Cauchy property \eqref{cauchy}.  If the poles of $dx$ dominate the poles of $dy$ then $dy/dx$ has poles only at the zeros $\cp_i$ of $dx$.  Then $\omega(p)\equiv 0$ since
$$\sum_{i=1}^N \Res_{p'=\cp_i}\frac{dy}{dx}(p')B(p,p')=-\Res_{p'=p}\frac{dy}{dx}(p')B(p,p')=-d\left(\frac{dy}{dx}(p)\right)
$$
and hence it is invariant under each local involution $\sigma_i$.  Since the Riemann surface $\Sigma$ is compact it automatically satisfies \eqref{shape} hence the claim is proven.
\end{proof}
\begin{remark}
In fact Corollary~\ref{spec2cohft} allows a weaker hypothesis which we will need.  We can instead allow $dy$ to be a locally defined meromorphic differential, essentially a connection, which is ambiguous up to $dy\mapsto dy+\lambda dx$.  The conclusion of Corollary~\ref{spec2cohft} still holds since $d\left(\frac{dy}{dx}\right)$ is globally defined.
\end{remark}

\section{Hurwitz Frobenius manifolds}   \label{sec:hur}

In this section we first remind the reader of Dubrovin's construction of a Frobenius manifold on a cover of Hurwitz space and then prove a number of deformation lemmas, which will be useful in the following sections.

\subsection{Dubrovin's construction}   \label{sec:dubcon}

 As defined in the introduction, denote by $\widetilde{H}_{g,\mu}$ the moduli space of tuples $(\Sigma,x,\{ {\cal A}_i,{\cal B}_i\}_{i=1,...,g})$ consisting of covers $x:\Sigma\to\mathbb{P}^1$ of genus $g$ with fixed ramification profile above infinity $\mu=(\mu_1,\dots,\mu_n)$ together with a choice of a symplectic basis of cycles $({\cal A}_i,{\cal B}_i)_{i=1, \dots,g}$ and marked branches of $x$ at each point above $\infty$.  

Given such a generic cover $x$, we denote its simple branch points 
$$
\forall i=1,\dots, N \, , \; u_i = x(\cp_i)  \qquad \hbox{where} \qquad \left. d x \right|_{\cp_i} = 0.
$$
By the Riemann-Hurwitz formula:
$$
N = 2g -2+n+|\mu| 
$$
and since an element of the Hurwitz space is defined up to a finite information by its critical values, this gives the dimension 
$$
\dim_{\mathbb{C}} \widetilde{H}_{g,\mu}  = 2g -2+n+|\mu| .
$$
In the introduction we claimed that 
\beq \label{flat=can}
\#\cd=N=\#\{p\mid dx(p)=0\}
\eeq
i.e. the number of generalised contours, defined by $\cd$ in \eqref{gencon}, coincides with $\dim_{\mathbb{C}} \widetilde{H}_{g,\mu}$.  This follows from the fact that $dx$ is a meromorphic differential so its divisior $(dx)=Z-P$ has degree $2g-2$, where $Z$ and $P$ are the zeros and poles of $dx$.  Hence 
$$\dim_{\mathbb{C}} \widetilde{H}_{g,\mu}=|Z|=2g-2+|P|=\dim H_1(\Sigma,x^{-1}(\infty))-1+\sum_{i=1}^d\mu_i=\#\cd.
$$
The last equality is clear since the elements of $\cd$ consist of firstly $\{x\ca_1,...,x\ca_g,\cb_1,...,\cb_g,\gamma_i, i=2,...,d\}$ which has cardinality equal to $\dim H_1(\Sigma,x^{-1}(\infty))=2g-1+d$, together with $-1+\sum_i\mu_i=|P|-d-1$ extra elements $x^{k/\mu_i}\cc_{\infty_i},k=1,...,\mu_i, i=1,...,d$ remove $x\cc_{\infty_1}$.

%\begin{remark} Let us just remark that $n$ can be decomposed in the dimension of space of complex structures on $\Sigma_g$, the positions of the $m+1$ pre-images of $\infty$ and the dimension of the space of meromorphic functions $x$ with the right divisor which can be computed by Riemann-Roch. \end{remark}

We use the critical values $u_i$ as local coordinates in an open dense domain of $\widetilde{H}_{g,\mu}^s\subset\widetilde{H}_{g,\mu}$ where $u_i \neq u_j$ for $i \neq j$. The vector fields $\partial_{u_i}$ give a basis of $T\widetilde{H}_{g,\mu}$ and define a multiplication $\cdot$ given by:
\beq\label{defmult}
 \partial_{u_i} \cdot \partial_{u_j} = \delta_{ij} \partial_{u_i}.
 \eeq
 We denote the unity and the Euler vector fields:
 \beq \label{defunity}
 e  = \sum_{i=1}^N \partial_{u_i},\quad
 E = \sum_{i=1}^N u_i \partial_{u_i}.
 \eeq 
 Let us now define one-forms on $\widetilde{H}_{g,\mu}$. For any quadratic differential $Q$ on $\Sigma$, define the one-form
 $$
 \Omega_Q = \sum_{i=1}^N du_i \Res_{p = \cp_i} \frac{Q(p) }{ dx}.
$$

Dubrovin defines a set of differentials $\phi$ on $\Sigma$, defined in \eqref{primdif} and described in more detail below, which have poles dominated by the poles of $dx$.  They are known as {\em primary differentials} and used to produce a quadratic differential $Q=\phi^2$.

\bt[Dubrovin \cite{Dub94}]  \label{dubhur}
For any primary differential $\phi$, $\widetilde{H}_{g,\mu}^s\bigcap \{u \, | \, \phi(\cp_i) \neq 0\}$ is equipped with a structure of a Frobenius manifold with
 multiplication \eqref{defmult}, unity and Euler vector fields \eqref{defunity} and metric
% $$\left<\partial'\, , \, \partial''\right>_{\phi} := \Omega_{\phi^2}(\partial' \cdot \partial'').$$
\beq  \label{eta}
\eta:=\sum_{i=1}^N du_i^2\cdot \Res_{p= \cp_i} \frac{\phi^2(p) }{ dx(p)} =\sum_{i=1}^N du_i^2\cdot\phi(\cp_i)^2
\eeq
where we used the notation of definition \ref{evaluationform} for the evaluation of a one-form at a point. In addition, the corresponding flat coordinates $\left(t_{\alpha}\right)_{\alpha = 1,\dots , N}$ can be explicitly written
in terms of periods of $\phi$ via
$$t_{\alpha}=\int_{\cc_\alpha}\phi$$
for any $\cc_\alpha \in\cd$.
%$$y(p) = \hbox{v. p.} \, \int_{\infty_1}^p \eta $$ with $\infty_1$ the pole of degree $\mu_1$ of the map $x$.

%In addition, $x = x(y)$ is the superpotential of this Frobenius manifold.
\et

This means that the data of such a Frobenius manifold structure on $\widetilde{H}_{g,\mu}$ is given by the choice of a primary differential $\phi$.  The definition of a primary differential uses the Torelli marking of $\Sigma$ as follows. Fix a point in $\widetilde{H}_{g,\mu}$, i.e. a pair $(\Sigma,x)$ (a point in a Hurwitz space) together with a basis $\left({\cal A}_i,{\cal B}_i\right)_{i=1,\dots,g}$ (a sheet of $\widetilde{H}_{g,\mu}$ seen as a cover). 
%In the next section, we will express these primary differentials in terms of cycles but for the moment, let us remark that, in this expression, $p$ is fixed since it given by a choice of Frobenius structure, i.e. a choice of $\eta$. 
%For flat coordinates the differentials $\phi_{\partial_t}$ are the one defined by Dubrovin.
Recall from the introduction that there is a unique Bergman Kernel $B(p,p')$ which is a bidifferential of the second kind normalised to have zero periods over the $\cal A$-cycles in the Torelli marking $\{ {\cal A}_i,{\cal B}_i\}$.  For any generalised contour $\cc_{\alpha}\in\cd$ we define a primary differential by 
$$
\phi_{\alpha}(p)=dy_{\alpha}(p)=\oint_{\cc^*_{\alpha}}B(p,p')
$$
which is locally holomorphic on $\Sigma\setminus x^{-1}(\infty)$.  Here, the dual $\cc^*_{\alpha}=\eta_{\alpha\beta}\cc_{\beta}$ with respect to the metric $\eta$.

Following Dubrovin, let us classify these cycles in 5 types:
\begin{itemize}

\item Type (1): for $i=1,\dots,d$ and $k = 1,\dots,\mu_i-1$ : \label{dubclass}
$$
\int_{p \in { \cal C}_{t_{i,k}}} f(p) = \frac{1 }{ \mu_i-1} \Res_{p \to \infty_i} x(p)^{\frac{k }{ \mu_i}} f(p);
$$

\item Type (2) : for $i=2,\dots, d$:
$$
\int_{p \in { \cal C}_{v_i}} f(p) =  \Res_{p \to \infty_i} x(p) \,  f(p);
$$

\item Type (3):  for $i=2,\dots, d$:
$$
\int_{p \in { \cal C}_{w_i}} f(p) = \hbox{v.p.} \int_{\infty_1}^{\infty_i} f(p) ;
$$

\item Type (4): for $i=1,\dots,g$:
$$
\int_{p \in { \cal C}_{r_i}} f(p) = - \oint_{{\cal A}_i}x(p) \,  f(p) ;
$$

\item Type (5): for $i=1,\dots,g$:
$$
\int_{p \in { \cal C}_{s_i}} f(p) = \frac{1 }{ 2 i \pi} \oint_{{\cal B}_i}   f(p) .
$$

\end{itemize}

We see that the two important systems of coordinates---flat coordinates and canonical coordinates---correspond to cycles in $\cd$, respectively zeros of $dx$.  These sets have the same cardinality by \eqref{flat=can}.

\subsection{Vector fields, cycles and meromorphic differentials.}  \label{vecycdif}

Let us now introduce a correspondence between vector fields and meromorphic forms using the Bergman kernel $B$ which allows us to express all the quantities defining the Hurwitz Frobenius manifold in terms of meromorphic forms.  For flat coordinates
$$
\partial_{t_{\alpha}} \mapsto\phi_{\alpha}(p) = \int_{p' \in \cc^*_{\alpha}} B(p,p').%=-\partial_{t_{\alpha}}ydx.
$$
By linearity, for any vector field $v$, we can define a cycle $\cc_v$ by
\beq  \label{vec2cyc}
\cc_v = \sum_{\alpha} \left<v,\partial_{t_{\alpha}}\right>_{\eta} \cc_{\alpha}
\eeq
a meromorphic differential
$$
\phi_v(p) = \int_{p' \in \cc_v} B(p,p')
$$
and the metric $\eta$ by
\beq  \label{etadiff}
\left<v_1,v_2\right>_{\phi} = \sum_{i} \Res_{p= \cp_i} \frac{\phi_{v_1} \phi_{v_2} }{ dx(p)} .
\eeq
Note that \eqref{vec2cyc} and \eqref{etadiff} are proven by verifying them on a basis.  We choose the flat basis, to prove \eqref{vec2cyc}.  Substitute $v=\partial_{t_{\alpha}}$ into \eqref{vec2cyc} to get $\cc_{\partial_{t_{\alpha}}} = \sum_{\beta} \left<\partial_{t_{\alpha}},\partial_{t_{\beta}}\right>_{\eta} \cc_{\beta}=\sum_{\beta} \eta_{\alpha\beta} \cc_{\beta}=\cc^*_{\alpha}$ as required.  We choose the canonical basis to prove \eqref{etadiff} as follows.

%Since $v=\sum_{\alpha,\beta}\left<v,\partial_{t_{\alpha}}\right>_{\eta}\eta^{\alpha\beta}\partial_{t_{\beta}}$ we have $\cc_v = \sum_{\alpha,\beta} \left<v,\partial_{t_{\alpha}}\right>_{\eta}\eta^{\alpha\beta}\cc^*_{\beta}$ hence

Apply \eqref{vec2cyc} to the canonical vector fields to get
$$ \cc_{\partial_{u_i}}=\sum_{\alpha}\left<\partial_{u_i},\partial_{t_{\alpha}}\right>_{\eta} \cc_{\alpha}=\sum_{\alpha}\phi(\cp_i)\Psi^i_{\alpha}\cc_{\alpha}
$$
and hence
$$ \phi_{\partial_{u_i}}(p)=\int_{ \cc_{\partial_{u_i}}}B(p,p')=\sum_{\alpha}\phi(\cp_i)\Psi^i_{\alpha}\int_{\cc_{\alpha}}B(p,p')=\sum_{\alpha,\beta}\phi(\cp_i)\Psi^i_{\alpha}\eta^{\alpha\beta}\int_{\cc^*_{\beta}}B(p,p')=\sum_{\alpha,\beta}\phi(\cp_i)\Psi^i_{\alpha}\eta^{\alpha\beta}\phi_{\beta}(p)
$$
We will study $\phi_{\partial_{u_i}}$ via evaluation at $\cp_j$.
$$ \phi_{\partial_{u_i}}(\cp_j)=\sum_{\alpha,\beta}\phi(\cp_i)\Psi^i_{\alpha}\eta^{\alpha\beta}\phi_{\beta}(\cp_j)=\sum_{\alpha,\beta}\phi(\cp_i)\Psi^i_{\alpha}\eta^{\alpha\beta}\Psi^j_{\beta}=\delta_{ij}\phi(\cp_i)$$
which uses the relation $\phi_{\beta}(\cp_j)=\Psi^j_{\beta}$ proven in Proposition~\ref{psisig}.
Since $\phi_{\partial_{u_i}}(p)$ vanishes at $\cp_j$ for $j\neq i$, \eqref{etadiff} becomes rather simple:
$$
\left<\partial_{u_i},\partial_{u_j}\right>_{\phi} = \sum_k \Res_{p= \cp_k} \frac{\phi_{\partial_{u_i}} \phi_{\partial_{u_j}} }{ dx(p)} =\delta_{i,j} \Res_{p= \cp_i} \frac{\phi^2(p)}{ dx(p)} 
$$
in agreement with \eqref{eta} and hence proving \eqref{etadiff} for all vector fields.

The product in terms of the canonical basis gives us a formula in terms of the matrix $\Psi$ of change of basis from flat to canonical which takes the form of the Verlinde formula, or Krichever formula depending on the context. (The Verlinde formula is actually for the degree 0 part of the theory.) This can be written for example following \cite{Dub94} equation (5.61) 
\beq  \label{threept}
C_{\alpha\beta\gamma} = \sum_i \Res_{p= \cp_i} \frac{\phi_{\alpha}(p) \phi_{\beta}(p) \phi_{\gamma}(p) }{ dx(p) \phi(p)}.
\eeq
This depends on the choice of Frobenius structure through $\phi$ which appears in the denominator and a point in the Frobenius manifold through the dependence on $x$.

Let us finally identify the identity and the Euler field. The consistency condition for the identity vector field $\un=\partial_{t_{\alpha_0}}$
$$
\left<\partial_{t_\alpha},\partial_{t_\beta}\right>_{\phi} = C_{\alpha\beta\alpha_0 }
$$
imposes 
$$
\phi_{\un}=\phi_{\alpha_0 } =  \phi
$$
and the Euler vector field
$$
\phi_E = -E\cdot ydx|_{x\text{ fixed}}=E\cdot x dy|_{y\text{ fixed}}=x dy=x\phi
$$
uses variations of the structures which are described below.

\subsection{Rauch variational formula} \label{sec:rauch}
An important tool used in this paper is  Rauch variational formula expressing the variation of the Bergman kernel with respect to the position of the critical values.
\begin{equation}   \label{Rauch}
\frac{\partial}{\partial u_i}B(p_1,p_2)=\Res_{p=\cp_i}\frac{B(p,p_1)B(p,p_2)}{dx(p)}.
\end{equation}
Rauch originally derived the dependence of the Riemann matrix of periods of a Hurwitz cover on the critical values of the covering map in \cite{Rauch}. It later led to the expression of the variation of the Bergman kernel  in \cite{KKoTau}.

In the present context, the meaning of the variation is as follows.  Over the Frobenius manifold $M=\widetilde{H}_{g,\mu}$ we have a universal curve $\pi:\tilde{C}\to M$ and a function $x:\tilde{C}\to M\times\overline{\bc}$ satisfying:
\begin{itemize}
\item[(i)] Each fibre $C=C_u=\pi^{-1}(u)$ is a Riemann surface.
\item[(ii)] $x$ is meromorphic on each fibre $C$.
\item[(iii)] The critical values $\{u_1,...,u_n\}$ of $x$ on each fibre above a semi-simple point are canonical coordinates for $M$.
\end{itemize}
For any vector field $\partial\in\Gamma(TM)$, we choose a lift $\tilde{\partial}\in\Gamma(TC)$ so that $\tilde{\partial} x=0$.  We abuse terminology and write $\tilde{\partial}=\partial$.  Hence we make sense of variations of a function $f(p_1,p_2)$ on $C\times C$ by identifying $p_i\in C_u$ with $p_i'\in C_{u'}$ when $x(p_1)=x(p_1')$.

Rauch's variational formula for the Bergman kernel leads to variational formulae for other quantities, in particular primary differentials.
\beq \label{vardy}
\partial_{u_i}dy(p)=\partial_{u_i}\int_{\cc_{\alpha}}B(p,p')=\int_{\cc_{\alpha}}\partial_{u_i}B(p,p')=\int_{\cc_{\alpha}}B(p,\cp_i)B(p',\cp_i)=dy(\cp_i)B(p,\cp_i) .
\eeq
We apply this to give a short proof of flatness of the metric \eqref{eta} and refer to \cite{Dub94} for the full proof of Theorem~\ref{dubhur} which gives a different proof of flatness. The tangent space to $\widetilde{H}_{g,\mu}$ is spanned by primary differentials constructed from contours in $\cd$.  Hence the following lemma proves flatness of the metric.
\begin{lemma}
When $\cc,\cc'\in\cd$ then $\langle\phi_{\cc},\phi_{\cc'}\rangle_{\phi}$ is constant in $\{u_1,...,u_N\}$.  
\end{lemma}
\begin{proof} From the Rauch's variational formula \eqref{Rauch}, we have
$\partial_{u_i}\phi_{\cc}(p)=\phi_{\cc}(\cp_i)B(p,\cp_i)$.  This uses the fact that the contour $\cc$ depends only on a geometric contour independent of the choice of $u_i$, and possibly a function of $x$ which is constant, i.e. $\partial_{u_i}x=0$ by assumption.  Hence
$$\partial_{u_j}\langle\phi_{\cc},\phi_{\cc'}\rangle=\sum_i\Res_{p=\cp_i}\frac{\partial_{u_j}(\phi_{\cc}(p)\phi_{\cc'}(p))}{dx(p)}=\sum_i\Res_{p=\cp_i}\frac{B(p,\cp_j)(\phi_{\cc}(\cp_j)\phi_{\cc'}(p)+\phi_{\cc}(p)\phi_{\cc'}(\cp_j))}{dx(p)}=0.$$
Note that the integrand potentially has poles at $\cp_j$ and $\infty_k$ but since each $\phi_{\cc}(p)$ is dominated by $dx(p)$ at each $p=\infty_k$ the poles at $\infty_k$ are removable.  Hence the last equality uses the fact that the integrand has poles only at $\cp_i$, $i=1,...,N$ so that the sum of its residues at $\cp_i$ is 0.  
\end{proof}

The following theorem proven by Shramchenko identifies the $R(z)$ matrix of the Hurwitz Frobenius manifold with the Laplace transform of the Bergmann kernel.  It uses the Rauch's variational formula.  
\bt[Shramchenko, \cite{ShrRie}]  \label{th:shr}
Given a point $\left(\Sigma,x, ({\cal A}_i,{\cal B}_i)_{i=1,\dots,g} \right)$ in the cover of a Hurwitz space with $B(p,p')$ normalised on the ${\cal A}$-cycles together with a choice of admissible differential $\phi$ the $R(z)$ matrix of the Hurwitz Frobenius manifold is given by:
\beq \label{RHur}
\left[R^{-1}(z)\right]_{j}^i := -  \frac{\sqrt{z}}{\sqrt{2\pi}} \int_{\Gamma_j} e^{-\frac{(x(p)-u_j)}{ z}}  \, B(p,\cp_i).
\eeq

\et

The resemblance of \eqref{RHur} and \eqref{eq:Identification-MatrixR} means we are now in a position to prove Theorem~\ref{th:main}. Let us also remark that Shramchenko's result goes further than a formal series  in $z$. Indeed, \cite{ShrRie} defines integration cycles $\Gamma_i$ such that $R(z)$ is the regular part of the expansion of a solution to Dubrovin's linear system which is well defined in a half plane.

\begin{proof}[Proof of Theorem~\ref{th:main}]
The proof combines Theorem~\ref{th:DOSS}, Theorem~\ref{thm:CompatibilityTest} and Theorem~\ref{th:shr}.  

Define the spectral curve $(\Sigma,x,y,B)$ by a generic point $(\Sigma,x)\in H_{g,\mu}$ equipped with a bidifferential $B$ normalised over a given set of %\textcolor{red}{geometric} 
$\cal A$-cycles, and a primary differential by $dy(p):=\oint_\cc B(p,p')$ for some $\cc\in\cd$, defined in \eqref{gencon}.  If the spectral curve satisfies the conditions \eqref{eq:Identification-MatrixR}-\eqref{shape} of Theorem~\ref{th:DOSS} for the $R(z)$ matrix of the Hurwitz Frobenius manifold then topological recursion applied to the spectral curve produces the ancestor invariants of the Frobenius manifold via the decomposition of $\omega_{g,n}$ given by \eqref{TRCohFT} and the theorem is proven.

By Theorem~\ref{th:shr} the $R(z)$ matrix of the Hurwitz Frobenius manifold is given by \eqref{RHur} hence condition~\eqref{eq:Identification-MatrixR} is satisfied.  Next we need to show that the choice of $y$ is the correct one.  But since $dy(p):=\oint_\cc B(p,p')$ the poles of $dy$ are dominated by the poles---the pole behaviour of the integrals over generalised cycles described in Section~\ref{sec:dubcon} is given in \cite{Dub94}---hence the spectral curve is dominant and Corollary~\ref{spec2cohft} applies, proving that condition~\eqref{eq:Identification-Dilaton} is satisfied.  Finally condition~\eqref{shape} is satisfied by Lemma~\ref{eynlem} since $\Sigma$ is compact and $x$ is meromorphic.

\end{proof}

\subsection{Shramchenko's deformation.}

Following methods of Kokotov-Korotkin \cite{KKoNew}, Shramchenko \cite{ShrDef} defined deformations of Dubrovin's Frobenius manifold structures on $\widetilde{H}_{g,\mu}$.  See also Buryak-Shadrin \cite{BShRem}.  Recall that once we are given $(\Sigma,x,\{ {\cal A}_i,{\cal B}_i\}_{i=1,...,g})$ and $\cd$, we define a Bergman kernel and use that to define primary differentials $\phi_{\alpha}$ for $\alpha\in\cd$.  Instead of the Bergman kernel $B(p,p')$ Shramchenko considered arbitrary Bergman kernels $\om_{0,2}^{[\kappa]}(p,p')$ on $\Sigma$ which is a symmetric bidifferential on $\Sigma\times\Sigma$, with a double pole on the diagonal of zero residue, double residue equal to $1$, and no further singularities. The set of such kernels is parameterised by symmetric matrices $\kappa$ of size $g \times g$. We denote by $\om_{0,2}^{[0]}=B$ the Bergman kernel normalised in the basis of cycles chosen, i.e.
$$
\forall i = 1,\dots,g \, , \; \oint_{{\cal A}_i} \om_{0,2}^{[0]} = 0.
$$
%Among all Bergman kernels normalised on geometric cycles, this choice is imposed by the choice of a sheet in $\widetilde{H}_{g,\mu}$ as a cover of some Hurwitz space.  
The key ingredients in the proofs of Theorems~\ref{th:main0} and \ref{th:main} are Rauch's variational principle for $B(p,p')$ which holds more generally for Bergman kernels normalised on geometric cycles and Eynard's formula \eqref{shape} which is valid for any $B=\om_{0,2}^{[\kappa]}$.

\begin{theorem}
The conclusion of Theorems~\ref{th:main0} and \ref{th:main} holds for Frobenius manifold structures on $\widetilde{H}_{g,\mu}$ defined by $\om_{0,2}^{[\kappa]}(p,p')$ when $\kappa$ is such that there exist a basis of geometric cycles $\left({\cal A}_i^{[\kappa]}, {\cal B}_i^{[\kappa]}\right)_i$ satisfying 
$$
\forall i = 1,\dots , g \, , \; \oint_{p' \in {\cal A}_i^{[\kappa]}} \om_{0,2}^{[\kappa]}(p,p') = 0.
$$
\end{theorem}

\subsection{Landau-Ginzburg model.}
In Section~\ref{vecycdif} we described a map from the tangent space at $p\in\widetilde{H}_{g,\mu}$ to the vector space spanned by primary differentials, denoted by $V^{\rm prim}_p$.  It was defined via a map to contours which are linear combinations of contours in $\cd$.  For $v\in T_p\widetilde{H}_{g,\mu}$ we defined
$$
v\mapsto\cc_v\mapsto\phi_v(p)
$$
A more direct path uses variations.  It is known as a Landau-Ginzburg model for $(\Sigma,x,dy)$ and defined by:
$$ \begin{array}{ccc}T_p\widetilde{H}_{g,\mu}&\to& V^{\rm prim}_p\\
v&\mapsto&v\cdot(-ydx)\end{array}.
$$
So the claim is that the variation gives the composition of the two maps above, i.e. $\cdot(-ydx(p))=\phi_{\cdot}(p)$.  We prove this relation in terms of flat coordinates. 
 \begin{lemma}  \label{varprimdif}
For $\cc_{\alpha}\in\cd$, the coordinate $t_{\alpha}=\int_{\cc_{\alpha}}dy$ is associated to the  differential  $\phi_{\alpha}(p)$ via 
$$\partial_{t_{\alpha}}[-y(p)dx(p)]=\phi_{\alpha}(p)=\int_{\cc_\alpha^*}B(p,p').$$
\end{lemma}
\begin{proof}
The main idea of the proof is to consider evaluation of $\partial_{t_{\alpha}}ydx(p)$ at $p=\cp_i$ in order to be able to integrate by parts.  From the variation of $dy$ with respect to canonical coordinates given in \eqref{vardy} we have
\beq  \label{varfdy}
\partial_{t_{\alpha}}dy(p)=\sum_i\Psi^i_{\alpha}\partial_{v_i}dy(p)=\sum_i\frac{\Psi^i_{\alpha}}{dy(\cp_i)}\partial_{u_i}dy(p)=\sum_i\frac{\Psi^i_{\alpha}}{dy(\cp_i)}dy(\cp_i)B(p,\cp_i)=\sum_i\Psi^i_{\alpha}B(p,\cp_i).
\eeq
Then
\begin{align*}
\partial_{t_{\alpha}}[-ydx](\cp_i)=-\Res_{p=\cp_i}\frac{1}{\sqrt{2(x(p)-u_i)}}\partial_{t_{\alpha}}[ydx]&
=\Res_{p=\cp_i}\sqrt{2(x(p)-u_i)}\partial_{t_{\alpha}}dy\\
&=\Res_{p=\cp_i}\sqrt{2(x(p)-u_i)}\sum_j\Psi^j_{\alpha}B(p,\cp_j)\\
&=\Res_{p=\cp_i}\sqrt{2(x(p)-u_i)}B(p,\cp_i)\Psi^i_{\alpha}\\
&=\Psi^i_{\alpha}=\phi_{\alpha}(\cp_i)
\end{align*}
where the second line uses \eqref{vardy}, the third line uses the fact that $B(p,\cp_j)$ has no pole at $\cp_i$ for $j\neq i$, the third line uses $\Res_{p=\cp_i}\sqrt{2(x(p)-u_i)}B(p,\cp_i)=1$, and the final equality uses Proposition~\ref{propsi}.

Hence 
$$\partial_{t_{\alpha}}[ydx](\cp_i)=\phi_{\alpha}(\cp_i),\quad i=1,...,N$$
which is nearly enough to guarantee that the differentials $\partial_{t_{\alpha}}ydx$ and $\phi_{\alpha}$ agree.   Define the function on $\Sigma$ by 
$$f(p)=\frac{\partial_{t_{\alpha}}ydx(p)-\phi_{\alpha}(p)}{dx(p)}.
$$
Then $f(p)$ has no poles since the numerator of $f(p)$ vanishes at $p=\cp_i$ and $dx(p)$ has simple zeros at $p=\cp_i$.  Also, from \eqref{varfdy} we see that $\partial_{t_{\alpha}}ydx$ has no poles at $x=\infty$ hence $\partial_{t_{\alpha}}ydx-\phi_{\alpha}$ has poles only at $x=\infty$, dominated by poles of $dx$, since this is true of $\phi_{\alpha}$.  In particular $f(p)$ has no poles at $x=\infty$.

Thus $f(p)=c$ constant and $\partial_{t_{\alpha}}ydx(p)=\phi_{\alpha}(p)+cdx(p)$.  In \cite{Dub94} Dubrovin proves that the differential $\phi_{\alpha}(p)$ is either strictly dominated by $dx$ at at least one point $\infty_i\in x^{-1}(\infty)$, in which case $f(\infty_i)=0$, or $\phi_{\alpha}(p)$ is a connection with ambiguity given by $cdx$ for any constant $c$.  Hence we may assume $c=0$ and the lemma is proven.
\end{proof}

We can now identify the transition matrix $\Psi$ between flat and normalised canonical vector fields in an elegant way.  Flat coordinates correspond to periods along generalised contours while canonical coordinates correspond to (finite) critical points of $x$.  The Bergman kernel allows a natural marriage of the two.
\begin{proposition}[\cite{ShrDef}]  \label{psisig}
The transition matrix $\Psi$ between flat and normalised canonical vector fields, defined in \eqref{psipde} is given by
$$
\Psi^i_{\alpha} =  \int_{p \in \cc^*_{\alpha}}  B(p,\cp_i)=\phi_{\alpha}(\cp_i).
$$
\end{proposition}  \label{propsi}
As usual the indices $i=1,...,N$ are associated to the canonical coordinates and $\alpha=1,...,N$ are associated to the flat coordinates.
\begin{proof}
We have
$$\partial_{u_i}t_{\alpha}=\partial_{u_i}\int_{\cc_{\alpha}}dy=\int_{\cc_{\alpha}}\partial_{u_i}dy=dy(\cp_i)\int_{\cc_{\alpha}}B(p,\cp_i)
$$
where the last equality uses \eqref{vardy}, hence
\beq\label{t2v}
\partial_{v_i}=\sum_{\alpha}\int_{\cc_{\alpha}}B(p,\cp_i)\; \partial_{t_\alpha}.
\eeq
Now
$$[\partial_{v_1},...,\partial_{v_N}]\Psi=[\partial_{t_1},...,\partial_{t_N}]
$$
and since $\Psi^T\Psi=\eta$, or $\Psi\eta^{-1}\Psi^T=I$
we have 
$$[\partial_{v_1},...,\partial_{v_N}]=[\partial_{t_1},...,\partial_{t_N}]\eta^{-1}\Psi^T
$$
hence 
$$\partial_{v_i}=\sum_{\alpha,\beta}\eta^{\alpha\beta}\Psi^i_{\beta}\cdot\partial_{t_\alpha}
$$
and comparing this with \eqref{t2v} we see that 
$$\sum_{\beta}\eta^{\alpha\beta}\Psi^i_{\beta}=\int_{\cc_{\alpha}}B(p,\cp_i)$$
so
$$\Psi^i_{\gamma}=\sum_{\alpha,\beta}\eta_{\gamma\alpha}\eta^{\alpha\beta}\Psi^i_{\beta}=\sum_{\alpha}\eta_{\gamma\alpha}\int_{\cc_{\alpha}}B(p,\cp_i)
=\int_{\sum_{\alpha}\eta_{\gamma\alpha}\cc_{\alpha}}B(p,\cp_i)
=\int_{\cc^*_{\gamma}}B(p,\cp_i)
$$ 
as required.  The second equality in the statement of the proposition simply uses the definition $\phi_{\alpha}(p):=\int_{p \in \cc^*_{\alpha}}  B(p,p')$.
\end{proof}
\begin{remark}
The column of $\Psi$ corresponding to the vector field gives the square root of the diagonal coefficients $\eta_i^{1/2}$ of the metric $\eta$ in canonical coordinates.  From Proposition~\ref{psisig}, we have $\eta_i  = \phi(\cp_i)^2$ which agrees with \eqref{eta}.
\end{remark}
The transition matrix $\Psi$ gives rise to the $R$ matrix of the Frobenius manifold built from a choice of point $\left(\Sigma,x, ({\cal A}_i,{\cal B}_i)_{i=1,\dots,g} \right)$ in $\widetilde{H}_{g,\mu}$ given in Theorem~\ref{th:shr} together with a choice of admissible differential $\eta$.

\subsection{Flat coordinates}  \label{sec:flat}

Let us now explain how to recover the expression of the correlators of the CohFT in flat coordinates out of integration along contours in $\cd$. 

\begin{lemma}
For any generalised contour $\cc\in\cd$ and any $(g,n) \in \mathbb{N}\times \mathbb{N}^*$, the map 
$$\omega_{g,n}\mapsto\int_{\cc}\omega_{g,n}$$
defining the action of integration of the correlation functions is well defined.
\end{lemma}
\begin{proof}
Since $\cc$ is only an isotopy class of contours (with coefficients that are functions of $x$) in $\Sigma\setminus x^{-1}(\infty)$ and $\omega_{g,n}$ has poles in $\Sigma\setminus x^{-1}(\infty)$ we need to prove that the integral is independent of the choice of contour.  This is a consequence of the fact that $\omega_{g,n}$ and $x\omega_{g,n}$ have zero residues at $\cp_i$.  Note that the residues at $\infty_j$ might not be zero, but the contours are not allowed to deform through $\infty_j$.
\end{proof}
%Denote the action of a generalised contour $\cc$ by $\overline{\cc}$.  Note that any locally defined function $f$ around a point $\cp_i$ defines an action on $\omega_{g,n}$ by $\omega_{g,n}\mapsto\Res_{P_i}f\omega_{g,n}$.  Denote this action by $f\overline{\cc}_i$
\begin{proposition}  \label{thmRBI}
For any $\cc\in\cd$ define $\phi_{\cc}(p)=\int_{\cc}B(p,p')=df_{\cc}$ where $f_{\cc}$ is locally valued.  As operators acting on $\omega_{g,n}$ for $2g-2+n>0$, 
$$\sum_i\Res_{\cp_i}f_{\cc}\cdot=\int_{\cc}\cdot\quad$$
in other words, 
$$\sum_i\Res_{p=\cp_i}f_{\cc}(p)\omega_{g,n}(p,p_2,...,p_n)=\int_{p\in\cc}\omega_{g,n}(p,p_2,...,p_n).$$
\end{proposition}
\begin{proof}
Recall Riemann's bilinear relation.  For meromorphic differentials $\phi$ and $\omega$ such that $\phi$ is residueless
\beq  \label{RBI}
\sum_P\Res_P f\cdot\omega=\frac{1}{2\pi i}\sum_{j=1}^g \left[ \oint_{\ca_i}\phi\oint_{\cb_i}\omega-\oint_{\cb_i}\phi\oint_{\ca_i}\omega \right]
\eeq
where $df=\phi$ for a locally defined primitive $f$ and the sum is over all poles $P$ of $\phi$ and $\omega$.

Primary differentials $\phi_{\cc}$ of types 1, 2 and 5, with respect to the classification given in Section~\ref{dubclass}, are residueless so apply \eqref{RBI} to $\phi=\phi_{\cc}$ and $\omega=\omega_{g,n}$.

For $\cc=\cb_j$, $i=j,...,g$, $\phi_{\cc}(p)=\int_{\cb_j}B(p,p')=\theta_j=df_{\cc}$ ($f$ defined locally) is a holomorphic differential satisfying $\int_{\ca_k}\theta_j=2\pi i\cdot\delta_{jk}$. Then \eqref{RBI} becomes:
$$
\sum_i\Res_{\cp_i} f_{\cc}\cdot\omega_{g,n}=\frac{1}{2\pi i}\sum_{k=1}^g \left[ \oint_{\ca_k}\phi_{\cc}\oint_{\cb_k}\omega_{g,n}-\oint_{\cb_k}\phi_{\cc}\oint_{\ca_k}\omega_{g,n} \right]
=\oint_{\cb_j}\omega_{g,n}=\oint_{\cc}\omega_{g,n}
$$
since $\oint_{\ca_k}\omega_{g,n}=0$.

For $\cc=x^{k/(n_i+1)}\cc_i$, $k=1,...,n_i+1$, $i=1,...,d$,  (which this includes both types 1 and 2) then $\phi_{\cc}=\Res_{\infty_i}x^{k/(n_i+1)}B=df_{\cc}$ is residueless and normalised so that $\int_{\ca_i}\phi_{\cc}=0$.  Then \eqref{RBI} becomes
\begin{align*}
\sum_{P=\cp_k,\infty_{\ell}}\Res_{P} f_{\cc}\cdot\omega_{g,n}&=\frac{1}{2\pi i}\sum_{j=1}^g \left[ \oint_{\ca_j}\phi_{\cc}\oint_{\cb_j}\omega_{g,n}-\oint_{\cb_j}\phi_{\cc}\oint_{\ca_j}\omega_{g,n} \right]=0\\
\Rightarrow \sum_k\Res_{\cp_k} f_{\cc}\cdot\omega_{g,n}&=-\Res_{\infty_i} f_{\cc}\cdot\omega_{g,n}=\oint_{\cc}\omega_{g,n}
\end{align*}
where the last equality uses the fact that $f_{\cc}\sim -x^{k/(n_i+1)}$ near $\infty_i$.

For $\cc=x\cc_i$, $i=1,...,d$, $\phi_{\cc}=\int_{\Gamma_i}B$ is a differential of the 3rd kind with simple poles at $\infty_1$ and $\infty_i$ normalised so that $\int_{\ca_k}\phi_{\cc}=0$.  Since $\omega_{g,n}$ is residueless we switch the roles of $\phi$ and $\omega$ in \eqref{RBI}.  Choose $F_{g,n}$ such that $dF_{g,n}=\omega_{g,n}$, i.e. a primitive with respect to one variable.  Then
\begin{align*}
\sum_{p=\cp_k,\infty_{\ell}}\Res_pF_{g,n}(p,p_2,...,p_n)\phi_{\cc}(p)&=\frac{1}{2\pi i}\sum_{j=1}^g \left[\oint_{\ca_j}\phi_{\cc}\oint_{\cb_j}\omega_{g,n}-\oint_{\cb_j}\phi_{\cc}\oint_{\ca_j}\omega_{g,n}\right]=0\\
\Rightarrow \sum_k\Res_{\cp_k} f_{\cc}(p)\omega_{g,n}(p,p_2,...,p_n)&=-\sum_k\Res_{\cp_k} F_{g,n}(p,p_2,...,p_n)\phi_{\cc}(p)\\
&=\Res_{p=\infty_i} F_{g,n}(p,p_2,...,p_n)\phi_{\cc}(p)+\Res_{p=\infty_1} F_{g,n}(p,p_2,...,p_n)\phi_{\cc}(p)\\
&=F_{g,n}(\infty_i,p_2,...,p_n)-F_{g,n}(\infty_1,p_2,...,p_n)\\
&=\int^{\infty_i}_{\infty_1}\omega_{g,n}=\int_{\cc}\omega_{g,n}
\end{align*}

For $\cc=x\ca_i$, $i=1,...,g$, we cannot apply \eqref{RBI} directly since $\phi_{\cc}$ is not a globally defined differential.  Instead we need to apply the proof of \eqref{RBI} as follows.   Cut $\Sigma$ along $\ca$ and $\cb$ cycles meeting at a common point $P_0$ to leave a simply-connected region $R\subset\Sigma$ on which $\phi_{\cc}$ and a primitive (with respect to one variable) $F_{g,n}(p)$ of $\omega_{g,n}(p)$ (suppress variables $p_2,...,p_n$) are well-defined.  As in the proof of \eqref{RBI} integrate $\frac{1}{2\pi i}\phi_{\cc}F_{g,n}$ along the boundary of $R$ given by the $\ca$ and $\cb$ cycles to get
\begin{align*}
\sum_{p=\cp_k}\Res_pF_{g,n}(p)\phi_{\cc}(p)&=\sum_{j=1}^g \left[\oint_{\ca_j}\phi_{\cc}\oint_{\cb_j}\omega_{g,n}-\oint_{\cb_j}\phi_{\cc}\oint_{\ca_j}\omega_{g,n}\right] -\int_{P_0+\cb_i}^{P_0+\cb_i+\ca_i}F_{g,n}(p)dx(p)\\
&=-\int_{P_0+\cb_i}^{P_0+\cb_i+\ca_i}F_{g,n}(p)dx(p)\\
&=-\int_{\ca_i}x(p)\omega_{g,n}(p)=-\int_{\cc}\omega_{g,n}(p)\\
\Rightarrow\sum_{p=\cp_k}\Res_pf_{\cc}(p)\omega_{g,n}(p)&=\int_{\cc}\omega_{g,n}(p).
\end{align*}

\end{proof}

Theorem~\ref{th:main} proved that topological recursion applied to the spectral curve $\left(\Sigma,x, ({\cal A}_i,{\cal B}_i)_{i=1,\dots,g} \right)$ with a choice of admissible differential $\phi=dy$ and $\om_{0,2} = B$, stores the ancestor invariants of the Hurwitz Frobenius manifold and hence proves Theorem~\ref{th:main0}.  We now prove the remainder of the statement of Theorem~\ref{th:main0} by showing how to extract the ancestor invariants via integration over generalised contours.
\begin{proposition}  \label{maincont}
Integration over flat contours $\cc_{\alpha}\in\cd$ produces primary invariants:
$$
\int_{\cc_{\alpha_1}}...\int_{\cc_{\alpha_n}}\omega_{g,n}=\int_{\overline{\modm}_{g,n}}I_{g,n}\Big(e^{\alpha_1}\otimes ...\otimes e^{\alpha_n}\Big)
$$
\end{proposition}
\begin{proof}
We will prove the dual statement
\beq  \label{cycleom}
\int_{\cc^*_{\alpha_1}}...\int_{\cc^*_{\alpha_n}}\omega_{g,n}=\int_{\overline{\modm}_{g,n}}I_{g,n}\Big(e_{\alpha_1}\otimes ...\otimes e_{\alpha_n}\Big).
\eeq
For $k>0$, 
$$\sum_i\Res_{\cp_i}y_{\alpha}\cdot V_k^i=0,\quad k>0
$$
where $V_k^i$ are defined in \eqref{Vdiff} and $dy_\alpha = \phi_\alpha$.  Hence the operator $\sum_i\Res_{\cp_i}y_{\alpha}\cdot$ only detects coefficients of $V_0^i(p)=B(\cp_i,p)$ in $\omega_{g,n}$ which stores the primary invariants by \eqref{TRCohFT}.  Now
 
%The differential $V_0^i(p)$ is holomorphic on $\Sigma-\cp_i$ and has a second order pole at $\cp_i$ with zero residue there.  Locally it is $V_0^i(p)=ds/s^2$ for $s=\sqrt{2(x-u_i)}$
$$
\Res_{\cp_j}y_{\alpha}(p)\cdot V_0^i(p)=\Res_{\cp_j}y_{\alpha}(p)\cdot B(\cp_i,p)=\delta_{ij}\phi_{\alpha}(\cp_j)=\Psi^i_{\alpha}
$$
since $B$ acts as a Cauchy kernel which sends $y_{\alpha}$ to evaluation of $dy_{\alpha}$.
Hence $\displaystyle\sum_i\Res_{\cp_i}y_{\alpha}\cdot$ acts as insertion of the vector 
$$\Psi^i_{\alpha}\cdot\partial_{v_i}=\partial_{t_{\alpha}}=e_{\alpha}
$$
into the ancestor invariant.  Thus, using Proposition~\ref{thmRBI} we see that as an operator on $\omega_{g,n}$
$$\int_{\cc^*_{\alpha}}\cdot=\sum_i\Res_{\cp_i}y_{\alpha}\cdot$$
acts as insertion of the vector $e^{\alpha}$ into the ancestor invariant and in particular \eqref{cycleom} holds.  Note that since $\cc^*_{\alpha}=\sum_{\beta}\eta_{\alpha\beta}\cc_{\beta}$ is a constant linear combination of contours in $\cd$, then Proposition~\ref{thmRBI} applies also to $\cc^*_{\alpha}$.
\end{proof}
\begin{remark}
Let us apply Proposition~\ref{maincont}, or more precisely \eqref{cycleom}, to the simplest case of $\omega_{0,3}$ to get the following.
\begin{align*}
C_{\alpha\beta\gamma}&=\int_{\overline{\modm}_{0,3}}I_{0,3}\Big(e_{\alpha}\otimes e_{\beta}\otimes e_{\gamma}\Big)\\
&=\int_{\cc^*_{\alpha}}\int_{\cc^*_{\beta}}\int_{\cc^*_{\gamma}}\omega_{0,3}\\
&=\int_{\cc^*_{\alpha}}\int_{\cc^*_{\beta}}\int_{\cc^*_{\gamma}}\sum_i\Res_{p=\cp_i}\frac{B(p_1,p)B(p_2,p)B(p_3,p)}{dx(p)dy(p)}\\
&=\sum_i\Res_{p=\cp_i}\int_{\cc^*_{\alpha}}\int_{\cc^*_{\beta}}\int_{\cc^*_{\gamma}}\frac{B(p_1,p)B(p_2,p)B(p_3,p)}{dx(p)dy(p)}\\
&=\sum_i\Res_{p=\cp_i}\frac{\phi_{\alpha}(p)\phi_{\beta}(p)\phi_{\gamma}(p)}{dx(p)dy(p)}
\end{align*}
which agrees with \eqref{threept} as expected.  Here we have used the formula 
$$
\omega_{0,3}(p_1,p_2,p_3)=\sum_i\Res_{p=\cp_i}\frac{B(p_1,p)B(p_2,p)B(p_3,p)}{dx(p)dy(p)}
$$
proven in \cite{EOInv}.
\end{remark}
The following proposition generalises Theorem~\ref{th:main0}.
\begin{proposition}
There exist generalised contours $\cc_{\alpha,k}=p_k(x)\cc_{\alpha}$, for $\cc_{\alpha}\in\cd$ and $p_k(x)=x^k+...$ a monic polynomial of degree $k$ in $x$, so that the ancestor invariants, corresponding to $d_j\geq 0$, appear as periods. 
\beq
\int_{\cc_{\alpha_1,k_1}}...\int_{\cc_{\alpha_n,k_n}}\omega_{g,n}=\int_{\overline{\modm}_{g,n}}I_{g,n}\Big(e^{\alpha_1}\otimes ...\otimes e^{\alpha_n}\Big)\cdot\prod_{j=1}^n\psi_j^{k_j}.%\langle\prod_{j=1}^n\tau_{k_j}\Big(e^{\alpha_j}\Big)\rangle.
\eeq

\end{proposition}
\begin{proof}
Using integration by parts, we see that the contour $x^k\cc_i$ acts on the differential $V^i_k(p)$ by
$$\int_{x^k\cc_i}V^i_k(p)=\int_{\cc_i}V^i_0(p).
$$
Hence there exists a monic polynomial $p_k(x)=x^k+...$ of degree $k$ in $x$ such that
$$\int_{p_k(x)\cc_i}V^j_m(p)=\delta_{ij}\delta_{km}.
$$
Define $\cc^*_{\alpha,k}=p_k(x)\cc^*_{\alpha}$, for $\cc_{\alpha}\in\cd$ then we have
\beq
\int_{\cc^*_{\alpha_1,k_1}}...\int_{\cc^*_{\alpha_n,k_n}}\omega_{g,n}=\langle\prod_{j=1}^n\tau_{k_j}(e_{\alpha_j})\rangle.
\eeq
\end{proof}

\section{Topological recursion for compact spectral curves}  \label{sec:compact}
In this section we associate to a spectral curve a so-called $\hat{R}(z)$ matrix, which in the case of spectral curves lying in the image of the map \eqref{DOSS} from CohFTs to spectral curves, coincides with the $R(z)$ matrix of the Frobenius manifold.

\begin{definition}
Given a spectral curve $(\Sigma,x,y,B)$ define a formal series
$$\hat{R}(z)  =  \sum_{k=0}^\infty \hat{R}_k z^k$$
with coefficients $N \times N$ matrices where $N=$ number of zeros of $dx$ by
\beq \label{rhat}
\left[\hat{R}^{-1}(z)\right]_{j}^i := -  \frac{\sqrt{z}}{\sqrt{2\pi}} \int_{\Gamma_j} e^{-\frac{(x(p)-u_j) }{ z}}  \, B(p,\cp_i).
\eeq
\end{definition}
In fact  $\hat{R}(z)$ depends only on $(\Sigma,x,B)$.  This definition begins with the spectral curve and produces $\hat{R}(z)$ which reverses the direction of \eqref{eq:Identification-MatrixR} where one begins with a Frobenius manifold and its associated $R(z)$ and produces a spectral curve.  In general $\hat{R}(z)$ will not arise out of a Frobenius manifold.
\begin{remark} Note that $\left[\hat{R}^{-1}(z)\right]_{j}^i$ is well-defined for $i=j$ because the integrand has a pole of residue zero at $\cp_i$, so $\Gamma_i$ can be deformed to avoid $\cp_i$ in a well-defined manner.
\end{remark}
\begin{remark}
The paths $\Gamma_i$ were defined only locally in a neighbourhood of $\cp_i$ in Section~\ref{sec:DOSS}.  That is also sufficient here, because again we are only concerned with the asymptotic expansion of $\hat{R}(z)$ at $z=0$.  Nevertheless, we can choose paths along which $x/z\to\infty$ in both directions, such as a path of steepest descent of $-x/z$ so that the series $\hat{R}(z)$ converges. 
\end{remark}
%Let us sketch the proof of this lemma. One first shows that \textcolor{red}{Check the factors of 2. In particular in the definition of the $R$-matrix and the factorization formula.} \beq\label{eqd/dx} d \left(\frac{ \xi_{d}^{i}(p)}{dx(p)}\right) = - {(2d+1) \over 2} \xi_{d+1}^{i}(p)  - \sum_{k=1}^{\mathfrak b} \check{B}_{d,0}^{i,k} \, \xi_{0}^{k}(p) \eeq by the studying  the singularities and periods of the meromorphic forms on both sides of this equation where $$ \xi_{d}^{i}(p) := \Res_{p' =\cp_i} {\om_{0,2}^{[\kappa]}(p',p) \over \left(x(p')-u_i\right)^d \, \sqrt{2(x(p')-u_i)}} . $$ In this expression, one denotes $$ \check{B}^{i,j}(z_1,z_2) := \sum_{(k,l)\in \mathbb{N}^2} \check{B}_{k,l}^{i,j} z_1^{k} z_2^{l}. $$ Then, Laplace transforming this equation with respect to $p \in \Gamma_j$ gives a recursive equation for the coefficients of the expansion of $\check{B}^{i,j}(z_1,z_2)$ around $z_1 \to 0$. The latter can be solved by equation \eqref{shape2}. \textcolor{red}{Should I write the details of the computations?}

Let us denote $\left[\hat{R}(z)\right]^i_{j} = \sum_k \left[\hat{R}_k\right]^i_{j} z^{-k}$. In particular, one has \beq\label{defhatR1} \left[\hat{R}_1\right]^i_{j}  =  B_{0,0}^{i,j}= B(\cp_i,\cp_j) . \eeq

\subsection{Factorisation property}
\label{sec-factor}
On a compact spectral curve $\hat{R}(z)$ shares the symplectic property of any $R(z)$ associated to a Frobenius manifold.  This is proven below as a consequence of a factorisation formula for the (Laplace transform of the) bidifferential $B$ in terms of $\hat{R}(z)$.  The factorisation formula is also required in the proof of Corollary~\ref{spec2cohft}.

%We first recall %in section \ref{sec-factor} that the Laplace transform of a Bergman kernel defined on a Hurwitz cover takes the form of equation \eqref{shape}. This result, proved by Eynard \cite{EynardInv}, is a consequence of the fact that the primitive of a Bergman kernel is a Cauchy kernel. This implies a very particular form for the action of the differentiation with respect to $x$ on $B$ leading to an inductive definition of the coefficients $B_{k,l}^{i,j}$ of the expansion of the odd part of the Bergman kernel \beq B^{odd}(p,q) = \delta_{i,j} {d\sqrt{x(p)-u_i} \, d\sqrt{x(q)-u_i} \over (\sqrt{x(p)-u_i}-\sqrt{x(q)-u_i})^2} + \sum_{k,l} B_{k,l}^{i,j} (x(p) - u_i)^{k}  (x(p) - u_i)^{l } dx(p) dx(q) \eeq in a neighbourhood of $(x(p),x(q))=(u_i,u_j)$ in terms of only the data of $B_{k,0}^{i,j}$.

%In \cite{EynardInv}, he proved that the Laplace transform of the Bergman kernel satisfies a factorisation formula, whatever normalisation one chooses.
 \bl [Eynard, \cite{EynardInv}]  \label{eynlem}
 Whenever the spectral curve is a Hurwitz cover of $\mathbb{P}^1$ with $dx$ a meromorphic form with simple zeroes, $\hat{R}(z)$---defined in \eqref{rhat}---satisfies the symplectic condition
\beq\label{eq-symplectic}
\hat{R}(z) \hat{R}^T(-z) = Id.
\eeq
Furthermore, the Laplace transform of a Bergman kernel
$$
\check{B}^{i,j}(z_1,z_2)  =
\frac{ e^{ \frac{u_i }{ z_1}+ \frac{u_j}{ z_2}}}{2 \pi \sqrt{z_1z_2}} \int_{\Gamma_i}
\int_{\Gamma_j} B(p,p') e^{- \frac{x(p) }{ z_1}- \frac{x(p')}{ z_2}} .
$$
satisfies
\beq\label{shape2}
\check{B}^{i,j}(z_1,z_2) =    - \frac{\sum_{k=1}^{N} \left[\hat{R}^{-1}(z_1)\right]^k_i \left[\hat{R}^{-1}(z_2)\right]^k_j }{ z_1+z_2}.
\eeq
 \el
 This means that the coefficients $B_{k,l}^{i,j}$ of the expansion of the Bergman kernel around the branch points $\cp_i$ and $\cp_j$ can be defined recursively in terms of the initial data $B_{k,0}^{i,j}$.  We give a proof here that differs from the proof in \cite{EynardInv}.
 \begin{proof}
 We have
 \begin{align}  \label{unberg}
\sum_{i=1}^N\Res_{q=\cp_i}\frac{B(p,q)B(p',q)}{dx(q)}&=-\Res_{q=p}\frac{B(p,q)B(p',q)}{dx(q)}-\Res_{q=p'}\frac{B(p,q)B(p',q)}{dx(q)}\\ \nonumber
 &=-d_p\left(\frac{B(p,p')}{dx(p)}\right)-d_{p'}\left(\frac{B(p,p')}{dx(p')}\right)
 \end{align}
 where the first equality uses the fact that the only poles of the integrand are $\{p,p',\cp_i,i=1,...,N\}$, and the second equality uses the Cauchy formula \eqref{cauchy} satisfied by the Bergman kernel.  The Laplace transform of the LHS of \eqref{unberg} is
 \begin{align*} 
 \frac{ e^{ \frac{u_i }{ z_1}+ \frac{u_j}{ z_2}}}{2 \pi \sqrt{z_1z_2}} \int_{\Gamma_i}
\int_{\Gamma_j} e^{- \frac{x(p) }{ z_1}- \frac{x(p')}{ z_2}}\sum_{k=1}^N\Res_{q=\cp_k}\frac{B(p,q)B(p',q)}{dx(q)}&
=\sum_{k=1}^N\frac{ e^{ \frac{u_i }{ z_1}+ \frac{u_j}{ z_2}}}{2 \pi \sqrt{z_1z_2}} \int_{\Gamma_i}
e^{- \frac{x(p) }{ z_1}}B(p,\cp_k) \int_{\Gamma_j}e^{- \frac{x(p')}{ z_2}}B(p',\cp_k)\\
&=\sum_{k=1}^N\frac{\left[\hat{R}^{-1}(z_1)\right]^k_i\left[\hat{R}^{-1}(z_2)\right]_{j}^k}{ z_1z_2}
 \end{align*}
and the Laplace transform of the RHS of \eqref{unberg} is
$$ -\frac{ e^{ \frac{u_i }{ z_1}+ \frac{u_j}{ z_2}}}{2 \pi \sqrt{z_1z_2}} \int_{\Gamma_i}
\int_{\Gamma_j} e^{- \frac{x(p) }{ z_1}- \frac{x(p')}{ z_2}}
\left\{d_p\left(\frac{B(p,p')}{dx(p)}\right)+d_{p'}\left(\frac{B(p,p')}{dx(p')}\right)\right\}=-\left(\frac{1}{z_1}+\frac{1}{z_2}\right)\check{B}^{i,j}(z_1,z_2)
 $$
since the Laplace transform satisfies
$$\int_{\Gamma_i} d\left(\frac{\omega(p)}{dx(p)}\right)e^{- \frac{x(p) }{ z}}=\frac{1}{z}\int_{\Gamma_i} \omega(p)e^{- \frac{x(p) }{ z}}.
$$
for any differential $\omega(p)$, by integration by parts.  Hence we see that the Laplace transform of \eqref{unberg} gives \eqref{shape2} as required.  Then \eqref{eq-symplectic} is a consequence of \eqref{shape2} and the finiteness of $\check{B}^{i,j}(z_1,z_2)$ at $z_2=-z_1$.
 \end{proof}

%\subsection{Topological recursion and Hurwitz Frobenius manifolds.}

%\noindent The proof follows from a few steps.

%Secondly, in section \ref{sec-Rauch}, we show that Rauch's variational formula implies a PDE satisfied by a matrix-valued power series $\hat{R}[z]$, built out of the Bergman kernel, together with a homogeneity condition with respect to the Euler vector field fixing an ambiguity in the definition of the diagonal terms.   The same PDE and homogeneity condition are satisfied by $R$.

%Thirdly, in section \ref{sec-endproof}, we identify these equations with the equation defining by induction the $R$-matrix of the Frobenius manifold structure built on $\widetilde{H}_{g,\mu}$ and conclude the proof by matching the initial data for the two sides.
\subsection{Defining equation for $\hat{R}_{ij}(z)$}

\label{sec-Rauch}

In the preceding section we defined an $\hat{R}$-matrix from which is equivalent to the Bergman kernel.  When the Bergman kernel is normalised on a basis of geometric cycles, we can go further and compute all the terms $\left[\hat{R}_k\right]_{ij}$ in terms of some minimal quantities.  This uses the Rauch variational formula \eqref{Rauch} which allows us to derive an equation for the $\hat{R}$ matrix.

When a spectral curve lies in the image of the map \eqref{DOSS} from CohFTs to spectral curves, the following theorem is a consequence of the properties \eqref{eqdefR2}, \eqref{unR}, \eqref{homR} of the $R$ matrix of a CohFT.  For compact spectral curves, by Theorem~\ref{th:shr} generically $\hat{R}=R$, but $\hat{R}$ is a little more general, and for example exists when critical values $u_i$ coincide and $R$ is problematic, since it is defined over the semi-simple part of the Frobenius manifold.  The outcome of the following theorem is that $\hat{R}$ resembles $R$ and it can be used to give an alternative proof of Theorem~\ref{th:main}.

\begin{theorem}\label{th:rhatPDE}
Given a triple $(\Sigma,x,B)$ consisting of a compact Riemann surface $\Sigma$, a meromorphic function $x:\Sigma\to\bc$ with zeros of $dx$ simple, and a Bergman kernel $B$, then $\hat{R}(z)$ 
satisfies \eqref{eqdefR2}, \eqref{unR}, \eqref{homR}, i.e.
\begin{align}
 \label{PDER^}
d\hat{R}(z) &= \frac{\left[\hat{R}(z),dU\right] }{ z} - \hat{R}(z)  \left[\Gamma,dU\right],
\\
\label{unR^}
\un\cdot \hat{R}(z)&=0 ,\\
\label{homR^}
(z\partial_z+E)\cdot \hat{R}(z)&=0 .
\end{align}
\end{theorem}
\begin{proof}
Although \eqref{unR^} is a consequence of \eqref{PDER^} we first prove  \eqref{unR^} and use this to prove \eqref{PDER^}.  \\

\noindent {\em Proof of} \eqref{unR^}: Differentiate Eynard's formula \eqref{shape2}
$$
\check{B}^{i,j}(z_1,z_2)  =
\frac{ e^{ \frac{u_i }{ z_1}+ \frac{u_j}{ z_2}}}{2 \pi \sqrt{z_1z_2}} \int_{\Gamma_i}
\int_{\Gamma_j}B(p,p') e^{- \frac{x(p) }{ z_1}- \frac{x(p')}{ z_2}} 
=\sum_{k=1}^{N}  \frac{\left[\hat{R}^{-1}(z_1)\right]_k^{i} \left[\hat{R}^{-1}(z_2)\right]_k^j }{ z_1+z_2}.
$$
to get
\begin{align*}
\sum_{k=1}^N\frac{\partial}{\partial u_k}\check{B}^{i,j}(z_1,z_2)&=\sum_{k=1}^N\frac{ e^{ \frac{u_i }{ z_1}+ \frac{u_j}{ z_2}}}{2 \pi \sqrt{z_1z_2}} \int_{\Gamma_i}
B(p,\cp_k) e^{- \frac{x(p) }{ z_1}}\int_{\Gamma_j}B(p',\cp_k)e^{- \frac{x(p')}{ z_2}}+\left(\frac{1 }{ z_1}+ \frac{1}{ z_2}\right)\check{B}^{i,j}(z_1,z_2)\\
&=\sum_{k=1}^N\frac{\left[\hat{R}^{-1}(z_1)\right]_{k}^i\left[\hat{R}^{-1}(z_2)\right]_{j}^k}{ z_1z_2}
-\sum_{k=1}^{N}  \frac{\left[\hat{R}^{-1}(z_1)\right]_k^{i} \left[\hat{R}^{-1}(z_2)\right]_k^j }{ z_1z_2}=0.
\end{align*}
Since $\left[\hat{R}^{-1}(z_1)\right]_j^{i} = - z\check{B}^{i,j}(z,0)$, we have 
$$\un\cdot\left[\hat{R}^{-1}(z_1)\right]_j^{i}=\sum_{k=1}^N\frac{\partial}{\partial u_k}\left[\hat{R}^{-1}(z_1)\right]_j^{i}= - z\sum_{k=1}^N\frac{\partial}{\partial u_k}\check{B}^{i,j}(z,0)=0,\quad \forall i,j$$
and since $\un\cdot \hat{R}(z)=0\Leftrightarrow\un\cdot \hat{R}^{-1}(z)=0$ this proves \eqref{unR^}.  \\

\noindent {\em Proof of} \eqref{PDER^}:  For $k\neq i$,
\begin{align}  \label{derR^}
{\partial \left[\hat{R}^{-1}(z)\right]^i_j \over \partial u_k}&= - {\partial\over \partial u_k}\frac{\sqrt{z}}{\sqrt{2\pi }} \int_{\Gamma_j} e^{-{(x(p)-u_j) \over z}}  \, B(p,\cp_i) \\ \nonumber
&=\delta_{k,j} \frac{1}{z}\left[\hat{R}^{-1}(z)\right]^i_j - \frac{\sqrt{z}}{\sqrt{2\pi }} \int_{\Gamma_j} e^{-{(x(p)-u_j) \over z}}  \, B(p,\cp_k) B(\cp_k,\cp_i)\\  \nonumber
&= \delta_{k,j}  { \left[\hat{R}^{-1}(z)\right]^i_j \over z}  + \left[\hat{R}^{-1}(z)\right]^k_j \beta_{ki} .
\end{align}
For $k=i$, by \eqref{unR^},
$${\partial \left[\hat{R}^{-1}(z)\right]^i_j \over \partial u_i}=-\sum_{m\neq i}{\partial \left[\hat{R}^{-1}(z)\right]^i_j \over \partial u_m}= (\delta_{i,j}-1){ \left[\hat{R}^{-1}(z)\right]^i_j \over z} - \sum_{m\neq i}  \left[\hat{R}^{-1}(z)\right]_j^m\beta_{mi} 
$$
which gives the $(i,j)$ component of the equation
$$
d\hat{R}^{-1}(z) = {\left[\hat{R}^{-1}(z),dU\right] \over z} +  \left[\Gamma,dU\right]\hat{R}^{-1}(z).
$$
Hence 
$$
d\hat{R}(z)=-\hat{R}(z)d\hat{R}^{-1}(z)\hat{R}(z) = -\hat{R}(z)\left({\left[\hat{R}^{-1}(z),dU\right] \over z} +  \left[\Gamma,dU\right]\hat{R}^{-1}(z)\right)\hat{R}(z)= {\left[\hat{R}(z),dU\right] \over z} - \hat{R}(z)  \left[\Gamma,dU\right]
$$
and \eqref{PDER^} holds.

\noindent {\em Proof of} \eqref{homR^}:  We begin with a variation of the proof of \eqref{shape2}, replacing the identity vector field with the Euler vector field.  We have
 \begin{align}   \label{eulberg}
\sum_{i} u_i {\partial \over \partial u_i}B(p,p')=\sum_{i=1}^Nu_i\Res_{q=\cp_i}\frac{B(p,q)B(p',q)}{dx(q)}&=\sum_{i=1}^N\Res_{q=\cp_i}\frac{x(q)B(p,q)B(p',q)}{dx(q)}\\\nonumber
&=-\Res_{q=p}\frac{x(q)B(p,q)B(p',q)}{dx(q)}-\Res_{q=p'}\frac{x(q)B(p,q)B(p',q)}{dx(q)}\\ \nonumber
 &=-d_p\left(\frac{x(p)B(p,p')}{dx(p)}\right)-d_{p'}\left(\frac{x(p')B(p,p')}{dx(p')}\right).
 \end{align}
Then
\begin{align*} 
 \left(1 + z_1 {\partial \over \partial z_1} + z_2 {\partial \over \partial z_2}   +\frac{u_i}{z_1}+\frac{u_j}{z_2}\right) \check{B}^{i,j}(z_1,z_2)&
 =\frac{ e^{ {u_i \over z_1}+ {u_j\over z_2}}}{2 \pi \sqrt{z_1z_2}} \left(z_1 {\partial \over \partial z_1} + z_2 {\partial \over \partial z_2}\right)\int_{\Gamma_i}
\int_{\Gamma_j} e^{- {x(p) \over z_1}- {x(p')\over z_2}}B(p,p')\\
&=-\frac{ e^{ {u_i \over z_1}+ {u_j\over z_2}}}{2 \pi \sqrt{z_1z_2}} \int_{\Gamma_i}
\int_{\Gamma_j} \sum_{k} u_k {\partial \over \partial u_k}B(p,p')\\
&=\left(\frac{u_i}{z_1}+\frac{u_j}{z_2} -\sum_{k} u_k {\partial \over \partial u_k}\right)\check{B}^{i,j}(z_1,z_2)
\end{align*} 
where the second equality uses \eqref{eulberg} and integration by parts to show that for any differential $\omega(p)$ the Laplace transform satisfies
$$\int_{\Gamma_i} e^{- {x(p) \over z}}d\left(\frac{x(p)\omega(p)}{dx(p)}\right)=%\int_{\Gamma_i}e^{- {x(p) \over z}} \left(\frac{x(p)\omega(p)}{z}\right)=
z\frac{d}{dz}\int_{\Gamma_i}e^{- {x(p) \over z}}\omega(p).
$$
Hence we are left with the following equation which is essentially the Laplace transform of \eqref{eulberg}:
\beq  \label{eulR^}
\left[1 + z_1 {\partial \over \partial z_1} + z_2 {\partial \over \partial z_2}  + \sum_{i} u_i {\partial \over \partial u_i} \right] \check{B}(z_1,z_2) = 0.
\eeq
We will now take the $z_2\to0$ limit of \eqref{eulR^}.  From Eynard's formula \eqref{shape2} we see that ${\partial \over \partial z_2} \check{B}(z_1,z_2)$ is well-defined at $z_2=0$, hence  $\displaystyle{\lim_{z_2\to0}}z_2\tfrac{\partial}{\partial z_2} \check{B}(z_1,z_2)=0$.  We also have $\check{B}^{j,k}(z_1,0)= - \frac{1}{z_1}\left[\hat{R}^{-1}(z_1)\right]^j_k$.  Thus the $z_2\to0$ limit of \eqref{eulR^} becomes
$$
0=\left[1 + z_1 {\partial \over \partial z_1}   + \sum_{i} u_i {\partial \over \partial u_i} \right] \frac{1}{z_1}\left[\hat{R}^{-1}(z_1)\right]^j_k=
 \frac{1}{z_1}\left[ z_1 {\partial \over \partial z_1}+\sum_{i} u_i {\partial \over \partial u_i} \right]\left[\hat{R}^{-1}(z_1)\right]^j_k
$$
which gives \eqref{homR^}.
\end{proof}

%\subsection{Identification of the equations} \label{sec-endproof}

%Theorem~\ref{th:rhatPDE} implies $\hat{R}$ is fully defined by the same equation as the $R$ matrix of a CohFT. In order to prove that the local expansion of the Bergman kernel agrees with the $R$ matrix of the Hurwitz Frobenius manifold of Section~\ref{sec:hur}, we only need to show that $\left[R_1\right]_{ij} = \left[\hat{R}_1\right]_{ij}$.

%This is exactly what is proven in \cite{KKoNew,ShrDef}.  It uses the identification of $R_1$ with rotation coefficients defined in \eqref{rotcoeff} and \eqref{rotcoeffberg}.  It is proven that\begin{equation}  \label{rotcoeffberg} \beta_{ij}=B(\cp_i,\cp_j) = {{B}_{0,0}^{i,j}} \end{equation} which agrees with \eqref{defhatR1}\footnote{A factor of $\frac{1}{2}$ in the formulae from \cite{KKoNew,ShrDef} is due to a different definition of evaluation of a differential at $\cp_i$.}. Rauch's variational principle can be used to show that $\beta_{ij}$ satisfy \eqref{rotcoeffPDE}. This proves that the $R$ and $\hat{R}$ matrices agree up to multiplication by a diagonal constant matrix.

%\bt  \label{th:main1} Given a point $\left(\Sigma,x, ({\cal A}_i,{\cal B}_i)_{i=1,\dots,g} \right)$ in the cover of a Hurwitz space together with a choice of admissible differential $\phi=dy$,  the topological recursion applied to the compact spectral curve $(\Sigma,x)$ with initial conditions $\om_{0,1} = ydx$ and $\om_{0,2}(p,p') = B(p,p')$ normalised on the ${\cal A}$-cycles produces the ancestor potential of the Frobenius manifold built out of the same data.\et

\begin{remark}
A spectral curve $(\Sigma,x,y,B)$ with $dy$ a primary differential is dominant---see Definition~\ref{def:dominant}---hence by Corollary~\ref{spec2cohft} it corresponds to a CohFT, which we have identified with the Hurwitz Frobenius manifold corresponding to primary differential.  More generally we can take $dy$ to be any linear combination of primary differentials, which is no longer a primary differential hence Theorem~\ref{th:main} does not apply, but the spectral curve is still dominant and hence corresponds to a CohFT.
\end{remark}

\section{Topological recursion for families of spectral curves}
Vector fields on the Frobenius manifold $\widetilde{H}_{g,\mu}$ can give rise to recursion relations between ancestor invariants.  In this section we show how the vector fields act on the multidifferentials $\omega_{g,n}$ arising out of  topological recursion and give rise to the recursion relations between ancestor invariants.

Over the Frobenius manifold $\widetilde{H}_{g,\mu}$ is a {\em universal curve} which is a family of spectral curves constructed via the underlying Hurwitz map $(\Sigma,x)$ together with natural cycles on $\Sigma$ used to define the full spectral curve $(\Sigma,x,y,B)$.  Note that topological recursion applied to a single spectral curve produces a CohFT which extends uniquely to a family of CohFTs, nicely encoded in a Frobenius manifold, and each giving rise to a corresponding spectral curve.  Hence in this way the family of spectral curves is reconstructed from any single spectral curve in the family.

%The same family of spectral curves arises from any single spectral curve in the family since 

Consider the family of multidifferentials $\omega_{g,n}$ obtained by applying topological recursion to the universal curve.  We can differentiate the multidifferentials $\omega_{g,n}$ with respect to vector fields on the Frobenius manifold $\widetilde{H}_{g,\mu}$.  As usual, for any vector field $v\in\Gamma(T\widetilde{H}_{g,\mu})$ we choose a lift $\tilde{v}\in\Gamma(TC)$ where $C$ is the universal curve over $\widetilde{H}_{g,\mu}$, so that $\tilde{v}\cdot x=0$.  We abuse terminology and write $\tilde{v}=v$.

First order deformations of topological recursion are described in \cite{EOInv}. There it is shown that deformations of $\omega_{0,1}$ propagate via the recursion to determine deformations of $\omega_{g,n}$.  Specifically, for $v$ a vector field on $\widetilde{H}_{g,\mu}$, if we can express the variation of $ydx$ as an integral of $B$ over a generalised contour $\cal C$, then the variation of $\omega_{g,n}$ uses the same contour as follows.
\beq \label{deform}
v\cdot ydx(p)=\int_\cc B(p',p)\quad\Rightarrow\quad v\cdot\omega_{g,n}(p_1,...,p_n)=\int_\cc\omega_{g,n+1}(p',p_1,...,p_n).
\eeq

Deformations with respect to natural vector fields on the Frobenius manifold correspond to relations between correlators in the CohFT.  In the remainder of this section we describe the dictionary between deformations by the unit and Euler vector fields and their realisations via topological recursion.

\subsection{Identity vector field.}
When $v=\un$ is the identity vector, we have
$$\un\cdot ydx|_{x\text{ fixed}} =-\un\cdot xdy|_{y\text{ fixed}}=-dy=-\Res_{p'=p}y(p')B(p,p')=\sum_P\Res_{p'=P}y(p')B(p,p')$$
where the sum is over the poles $P$ of $y$.  Hence by \eqref{deform}
\beq  \label{string}
\un\cdot \omega_{g,n} =\sum_P\Res_{p'=P}y(p')\omega_{g,n+1}=-\sum_i\Res_{p'=\cp_i}y(p')\omega_{g,n+1}.
\eeq
We can calculate the action of $\un$ on $\omega_{g,n}$ in a different way via the lift of $\un$ to the universal curve.  Note that there are flat coordinates $t_1,.,,.t_N$ such that $\un=\partial/\partial t_1$ where $t_1$ appears in $x$ as $x=x_0+t_1$ for $x_0$ independent of $t_1$.  The lift $\un$ necessarily annihilates $x$ so with respect to a local parameter $z$ on $\Sigma$
$$ 0=\un\cdot x=x'(z)\un\cdot z +1\quad\Rightarrow \quad \un\cdot z=-1/x'(z)
$$
where we used the explicit partial derivative $\partial_{t_1}x=1$.  Hence for any differential $\xi$ with no explicit $t_1$ dependence, locally $\xi=df$ so we have
$$\un\cdot \xi(z)=d\un\cdot f(z)=d(f'(z)\un\cdot z)=-d(f'(z)/x'(z))=-d(df/dx)=-d(\xi/dx).
$$
In other words the lift of $\un$ coincides on fibres with the operator
$$\un=-\frac{d}{dx}
$$
which acts on functions or differentials.  In particular, $\xi_k^{\alpha}$ have no explicit $t_1$ dependence so
$$\un\cdot\xi_k^{\alpha}=-d\left(\frac{\xi_k^{\alpha}}{dx}\right)=-\xi_{k+1}^{\alpha}.
$$
Furthermore $\omega_{g,n}$ has no explicit $t_1$ dependence since topological recursion is unchanged under $x\mapsto x+t_1$.  Hence
$$\un\cdot\omega_{g,n}=-\sum_{j=1}^nd\left(\frac{\omega_{g,n}(p_1,...,p_n)}{dx(p_j)}\right).
$$
The relation
$$
\sum_i\Res_{p'=\cp_i}y(p')\omega_{g,n+1}=-\sum_{j=1}^nd\left(\frac{\omega_{g,n}(p_1,...,p_n)}{dx(p_j)}\right)
$$
is proven in a different way in \cite{EOInv} as a direct consequence of topological recursion.  Here we have shown it to be a consequence of the action of the lift of the identity vector field on the universal curve.

The Hurwitz Frobenius manifold from Section~\ref{sec:hur} have flat identity hence the CohFT satisfies the pull-back relation \eqref{flatid}.  A consequence of \eqref{flatid} on correlators is known as the string equation which expressed in tau notation:
$$\left\langle \prod_{i=1}^n\tau_{k_i}(v_i)\right\rangle_g:=
\int_{\overline{\modm}_{g,n}}I_{g,n}(v_1\otimes\cdots\otimes v_n)\prod_{i=1}^n\psi_i^{k_i}$$
is given by
$$\left\langle \tau_0(1)\tau_{k_1}(v_1)\cdots \tau_{k_n}(v_n)\right\rangle_g=\sum_{i=1}^n\left\langle \tau_{k_1}(v_1)\cdots\tau_{k_i-1}(v_i)\cdots \tau_{k_n}(v_n)\right\rangle_g.$$
But this is precisely equivalent to the relation \eqref{string} since 
$$d\left(\frac{\xi_k^{\alpha}}{dx(p_j)}\right)=\xi_{k+1}^{\alpha}
$$
where coefficients of $\xi_k^{\alpha}$ correspond to insertions of the vector field $\partial/\partial t_{\alpha}$.

\subsection{Euler vector field.}

%\textcolor{red}{This is not quite correct.  It is subtle because $E$ acts on both the differentials and the coefficients.  It is possible to calculate the action of $E$ on $\partial/\partial t_{\alpha}$ and get the conformal homogeneity factor $d$.}
When $v=E$ is the Euler vector field, we have
\beq \label{euler}
E\cdot ydx|_{x\text{ fixed}} =-E\cdot xdy|_{y\text{ fixed}}=-xdy
=\Res_{p'=p}(\Phi-xy)(p')B(p,p')=-\sum_P\Res_{p'=P}(\Phi-xy)B(p,p')
\eeq
where $d\Phi=ydx$ and the sum is over the poles $P$ of $\Phi-xy$.  Hence
\begin{align} \label{eulerTR}
E\cdot \omega_{g,n} &=-\sum_P\Res_{p'=P}(\Phi-xy)\omega_{g,n+1}=\sum_i\Res_{p'=\cp_i}(\Phi-xy)\omega_{g,n+1}\\ \nonumber
&=(2g-2+n)\omega_{g,n}(p_1,...,p_n)-\sum_{j=1}^nd\left(\frac{x(p_j)\omega_{g,n}(p_1,...,p_n)}{dx(p_j)}\right)
\end{align}
where the last equality uses the dilaton and second string equation satisfied quite generally by the $\omega_{g,n}$, proven in \cite{EOInv}.  Analogous to the string equation above, which enables one to remove or insert the identity vector field, this last expression enables one to remove or insert the Euler vector field in correlators.  For example, in the Gromov-Witten case, it is given by the divisor equation.  

A conformal Frobenius manifold corresponds to a homogeneous CohFT.  A CohFT is {\em homogeneous} of weight $d$ if
\begin{equation}  \label{homog}
((g-1)d+n)I_{g,n}=\deg I_{g,n}(v_1\otimes...\otimes v_n)-\sum_{j=1}^n I_{g,n}(v_1\otimes...\otimes[E,v_j]\otimes... \otimes v_n)+\pi_*I_{g,n+1}(v_1\otimes... \otimes v_s\otimes E)
 \end{equation}
where $E$ is the Euler vector field and $\pi:\overline{\modm}_{g,n+1}\to\overline{\modm}_{g,n}$ is the forgetful map.  Equation \eqref{homog} allows one to 
remove or insert the Euler vector field in correlators and \eqref{euler} is equivalent to this relation.


\begin{thebibliography}{99}

\bibitem{ABOMod} J. E. Andersen, G. Borot and N. Orantin. 
\emph{Modular functors, cohomological field theories and topological recursion.}  
\href{http://arxiv.org/abs/1509.01387}{arXiv:1509.01387}

\bibitem{AtiTop} M.F. Atiyah
\emph{Topological quantum field theories.} 
Publications Math\'ematiques de l'IH\'ES {\bf 68} (1989), 175-186.

\bibitem{BouchardEynard} 
V.~Bouchard and B.~Eynard,
\emph{Think globally, compute locally.}
JHEP 02 (2013)143.

\bibitem{BShRem} A. Buryak and S.~Shadrin,
\emph{A remark on deformations of Hurwitz Frobenius manifolds.} 
Lett. Math. Phys. {\bf 93} no. 3, (2010), 243-252.

\bibitem{Dub94}
B.~Dubrovin,
\emph{Geometry of 2D topological field theories.}
Integrable Systems and Quantum Groups (Authors: R. Donagi, B. Dubrovin, E. Frenkel, E. Previato), Eds. M. Francaviglia, S. Greco, Springer Lecture Notes in Math. {\bf 1620} (1996), 120 -348.

\bibitem{Dub98}
B.~Dubrovin,
\emph{Painlev\'e transcendents and two-dimensional topological field theory.}
The Painlev\'e Property: One Century Later. R.Conte (Ed.), Springer Verlag, 1999, 287-412.

\bibitem{DNOPS}
P.~Dunin-Barkowski, P.~Norbury, N.~Orantin, A.~Popolitov, S.~Shadrin,
\emph{Dubrovin's superpotential as a global spectral curve.}
To appear in J. Inst. Math. Jussieu. \href{http://arxiv.org/abs/1509.06954}{arXiv:1509.06954}

\bibitem{DOSS12}
P.~Dunin-Barkowski, N.~Orantin, S.~Shadrin, L.~Spitz,
\emph{Identification of the Givental formula with the spectral curve topological recursion procedure.}
Comm. Math. Phys. {\bf 328}, (2014) 669-700. 

\bibitem{EynardInt}
B.~Eynard, 
\emph{Intersection numbers of spectral curves.}
\href{http://arxiv.org/abs/1104.0176}{arXiv:1104.0176}.


\bibitem{EynardInv}
B.~Eynard, 
\emph{Invariants of spectral curves and intersection theory of moduli spaces of complex curves.}
Communications in Number Theory and Physics {\bf 8} (2014), no. 3, 541-588.

\bibitem{EOInv}
B.~Eynard, N.~Orantin, 
\emph{Invariants of algebraic curves and topological expansion.}
Commun. Number Theory Phys. {\bf 1} (2007), no.~2, 347-452.

\bibitem{FSZTau}
C. Faber, S. Shadrin and D. Zvonkine,
\emph{Tautological relations and the $r$-spin Witten conjecture.}
Ann. Sci. \'{E}c. Norm. Sup\'{e}r. (4) {\bf 43} (2010), 621-658.

\bibitem{GiventalMain}
A.~Givental,
\emph{Gromov-Witten invariants and quantization of quadratic hamiltonians.}
\\
Mosc.~Math.~J. 1 (2001), no.~4, 551-568.

\bibitem{GiventalSem}
A.~Givental,
\emph{Semi-simple Frobenius structures in higher genus.} 
Int. Math. Res. Not. {\bf 23} (2001), 1265-1286.

\bibitem{KKoNew} A. Kokotov and D. Korotkin,
\emph{A new hierarchy of integrable systems associated to Hurwitz spaces.}
Phil. Trans. Royal Soc. A. {\bf 366} (2008), 1055-1088.

\bibitem{KKoTau} A. Kokotov and D. Korotkin,
\emph{Tau-functions on Hurwitz spaces.}
Math. Phys. Anal. Geom. {\bf 7} (2004), 47-96.

\bibitem{KriAlg} I.M. Krichever, 
\emph{Algebraic-geometric $n$-orthogonal curvilinear coordinate systems and solutions of the associativity equations.} 
Funct. Anal. Appl., {\bf 31}(1) (1997), 25-39.

\bibitem{ManFro} Y. Manin,
\emph{Frobenius manifolds, quantum cohomology, and moduli spaces.} 
Am. Math. Soc. Colloquium Publications, {\bf 47}. Am. Math. Soc., Providence, RI, 1999.

\bibitem{MilanovAncestor}
T. Milanov, 
\emph{The Eynard-Orantin recursion for the total ancestor potential.} 
Duke Math. J. {\bf 163} (2014), no. 9, 1795-1824.

\bibitem{PPZRel} R. Pandharipande, A. Pixton, and D. Zvonkine
\emph{Relations on $\overline{\modm}_{g,n}$ via 3-spin structures.}
J. Amer. Math. Soc. {\bf 28} (2015), no. 1, 279-309.

\bibitem{Rauch} H.E. Rauch,
\emph{Weierstrass points, branch points, and moduli of Riemann surfaces.} Comm. Pure Appl. Math. {\bf 12},  (1959), 543-560.

\bibitem{SaiPer} K. Saito,
\emph{Period Mapping Associated to a Primitive Form.}
Publ. RIMS, Kyoto Univ. {\bf 19} (1983), 1231-1264.

\bibitem{ShaBCO} S. Shadrin,
\emph{BCOV theory via Givental group action on cohomological fields theories.} Mosc. Math. J. {\bf9} (2), (2009), 411-429.

\bibitem{ShrDef} V. Shramchenko,
\emph{Deformations of Frobenius structures on Hurwitz spaces.}
Int. Math. Res. Not. {\bf 6} (2005),  339-387.

\bibitem{ShrRea} V. Shramchenko,
\emph{Real Doubles of Hurwitz Frobenius Manifolds.}
Comm. Math. Phys. {\bf 256}, (2005), 635-680.

\bibitem{ShrRie} V. Shramchenko,
\emph{Riemann-Hilbert problem associated to Frobenius manifold structures on Hurwitz spaces: irregular singularity.} 
Duke Math. J. {\bf 144} (2008), 1-52.

\bibitem{Teleman}
C.~Teleman,
\emph{The structure of 2D semi-simple field theories.}
Invent.~Math.~188, no.~3, 525-588.


\bibitem{ZhoFro}  Jian Zhou,
\emph{Frobenius manifolds, spectral curves and integrable hierarchies.}
\href{http://arxiv.org/abs/1512.05466}{arXiv:1512.05466}.

\end{thebibliography}
\end{document}